\documentclass[11pt,english]{smfart}

\usepackage{amsmath,leftidx,amsthm,sfheaders}
\usepackage[all]{xy}
\usepackage{xspace,smfthm}
\usepackage[psamsfonts]{amssymb}
\usepackage[latin1]{inputenc}
\usepackage{graphicx,color}
\usepackage{hyperref,fancyhdr}

\newtheorem*{remark}{\bf Remark}
\newtheorem*{question}{\bf Question}

\newtheorem{theorem}{\bf Theorem}[section]
\newtheorem{proposition}[theorem]{\bf Proposition}
\newtheorem{definition}[theorem]{\bf Definition}
\newtheorem{Theorem}{\bf Theorem}

\newtheorem*{claim}{\bf Claim}
\newtheorem{lemma}[theorem]{\bf Lemma}
\newtheorem{corollary}[theorem]{\bf Corollary}

\newtheorem{Corollary}[Theorem]{\bf Corollary}

\def\C{{\mathbb C}}
\def\N{{\mathbb N}}
\def\R{{\mathbb R}}

\def\B{\mathbb{B}}
\def\D{\mathbb{D}}

\def\J{\mathcal{J}}

\def\p{\mathbb{P}}

\def\supp{\textup{supp}}

\def\pe{\textup{ := }}
\def\rat{\textup{Rat}}
\def\Per{\textup{Per}}
\def\bif{\textup{bif}}

\def\Mand{\mathbf{M}}
\def\J{\mathcal{J}}

\def\un{{\underline{n}}}
\def\uw{{\underline{w}}}

\def\and{{\quad\text{and}\quad}}

\title{Distribution of postcritically finite polynomials II: speed of convergence}
\author{Thomas Gauthier $\&$ Gabriel Vigny}
\address{LAMFA, Universit\'e de Picardie Jules Verne, 33 rue Saint-Leu, 80039 AMIENS Cedex 1, FRANCE}
\email{thomas.gauthier@u-picardie.fr, gabriel.vigny@u-picardie.fr}

\begin{document}
\begin{abstract}
In the moduli space of degree $d$ polynomials, we prove the equidistribution of postcritically finite polynomials toward the bifurcation measure. More precisely, using complex analytic arguments and pluripotential theory, we prove the exponential speed of convergence for $\mathcal{C}^2$-observables. This improves results obtained  with arithmetic methods by Favre and Rivera-Letellier in the unicritical family 
and Favre and the first author in the space of degree $d$ polynomials. \\
We deduce from that the equidistribution of hyperbolic parameters with $(d-1)$ distinct attracting cycles of given multipliers toward the bifurcation measure
with exponential speed for $\mathcal{C}^1$-observables. As an application, we prove the equidistribution (up to an explicit extraction) of parameters with $(d-1)$ distinct cycles with prescribed multiplier  toward the bifurcation measure for any $(d-1)$ multipliers outside a pluripolar set.
\end{abstract}

\maketitle

\tableofcontents

\section{Introduction}

In a holomorphic family $(f_\lambda)_{\lambda\in\Lambda}$ of degree $d\geq2$ rational maps, Ma\~n\'e, Sad ans Sullivan~\cite{MSS} studied quite precisely the notion of $J$-stability. In particular, they related the instability of critical orbits under small perturbations of the dynamics and instability of periodic cycles (see also~\cite{lyubich}). 

This description has been enriched by DeMarco~\cite{DeMarco1} who introduced a current $T_\bif$ which is supported exactly on the bifurcation locus. This current and its self-intersections reveal to be an appropriate tool to study bifurcations from a measurable viewpoint. Now, consider the particular case of the moduli space $\mathcal{P}_d$ of critically marked complex polynomials of degree $d$ modulo conjugacy by
affine transformations. In that space, the maximal self-intersection of the bifurcation current induces a \emph{bifurcation measure} $\mu_\bif$, introduced by Bassanelli and Berteloot \cite{BB1}, which may be considered as the analogue of the harmonic measure of the Mandelbrot set when $d\geq 3$. The support of this measure is where maximal bifurcation phenomena occur. Recall that a polynomial is \emph{postcritically finite} if all its critical points have finite orbit, it is \emph{postcritically finite hyperbolic} if all its critical points are periodic cycle (the Julia set of such polynomials is then hyperbolic).
Recall that a polynomial is \emph{postcritically finite} if all its critical points have finite orbit.
 Using the bifurcation measure, it is proved in \cite{BB1} that  its support is accumulated by postcritically finite hyperbolic parameters (which are in a certain way the most stable parameters) and that it coincides with the closure of parameters having a maximal number of neutral cycles. Still using the bifurcation measure, the first author also showed in~\cite{Article1} that its support has maximal Hausdorff dimension.
 
As the bifurcation locus of the moduli space $\mathcal{P}_d$ is a  complicated fractal set, a natural approach is to study it on some dynamical slices. In particular, one can study the  maps having a superattracting orbit of fixed period $n$. The study of such a set is difficult and involves naturally arithmetic, combinatorics, topology and complex analysis (see e.g. \cite{MilnorCubic}). Furthermore, to understand the global geography of the moduli space $\mathcal{P}_d$, it is useful to approximate the bifurcation current (resp. the bifurcation measure) by dynamically defined hypersurfaces (resp. finite sets).  Following the topological description of the bifurcation locus given by~\cite{MSS}, one can try to approximate the bifurcation currents by different types of phenomena: existence of critical orbit relations or periodic cycles of given nature. 

~

\par
Let us focus first on the simplest case of the unicritical family, i.e. the family defined by $p_c(z):=z^d+c$, $c\in\C$. Consider the set $\Per(n):=\{c\in\C \, ; \ p_c^n(0)=0\}$ of parameters that admit a superattractive periodic point of period dividing $n$. Recall that the \emph{Multibrot set} $\Mand_d$ is defined by
$\Mand_d:=\{c\in\C \, ; \ \J_c \text{ is connected}\}$. Beware that $\Mand_d$ is a non-polar connected compact set. Finally, let $\mu_{\Mand_d}$ be the harmonic measure of $\Mand_d$.
Then, the first result in this direction goes back to Levin~\cite{Levin1}. In the quadratic family ($d=2$), he proved that the measure equidistributed on the set $\Per(n)$ converges to the harmonic measure of the Mandelbrot set, as $n\to\infty$. Favre and Rivera-Letelier~\cite{FRL} gave a \emph{quantitative} version of Levin's in the unicritical family, using arithmetic methods. Namely, they proved that there exists $C>0$ such that for any $n\geq1$ and $\varphi\in\mathcal{C}^1_c(\C)$:
$$\left|\frac{1}{d^{n-1}}\sum_{c\in\Per(n)}\varphi(c)-\int_\C\varphi\, \mu_{\Mand_d}\right|\leq C\left(\frac{n}{d^n}\right)^{\frac{1}{2}}\|\varphi\|_{\mathcal{C}^1}.$$
 Using complex analytic potential theory, we prove here the following.
\begin{Theorem}
Let $d\geq2$. Then, there exists a constant $C>0$ depending only on $d$ such that for any $\varphi\in\mathcal{C}^2_c(\C)$ and any $n\geq1$,
$$\left|\frac{1}{d^{n-1}}\sum_{c\in\Per(n)}\varphi(c)-\int_\C\varphi\, \mu_{\Mand_d}\right|\leq C\frac{n}{d^n}\|\varphi\|_{\mathcal{C}^2}.$$
\label{tm:Mandelbrot}
\end{Theorem}


A classical interpolation argument immediately gives Favre and Rivera-Letelier's above result.  We shall also give a similar estimate for non-postcritically finite hyperbolic parameters in Section~\ref{sec:unicritical}.

~

\par

More recently, this subject has been intensively studied in the moduli space $\mathcal{P}_d$ for $d\geq3$. As the critical points are marked, i.e. can be followed holomorphically in the whole moduli space $\mathcal{P}_d$, the bifurcation current decomposes as $T_\bif=\sum_{P'(c)=0} T_{c}$ (see \S \ref{sec:defbif} for the definition of $T_{c}$). For $n>k\geq 0$, let
\[\Per_c(n,k):=\{[P]\in\mathcal{P}_d~\, ; \ P^n(c)=P^k(c)\}~.\]
Refining Levin's approach, Dujardin and Favre prove in~\cite{favredujardin} that for any sequence $k(n)$ satisfying $0\leq k(n)<n$, the sequence $d^{-n}[\Per_c(n,k(n))]$ converges to $T_c$ in the weak sense of currents on $\mathcal{P}_d$. Recently, Okuyama~\cite{okuyama:distrib} gave a simplified proof of their result in the case $k(n)=0$. 

Compared to the unicritical family, a significant difficulty comes from the fact that two distinct critical points can very well be, at some parameter, in the immediate Fatou component of the same attractive periodic point. To overcome that, for $n\geq1$ and $w\in\C$, we let
\[\Per^*(n,w):=\{[P]\in\mathcal{P}_d \, ; \ P \text{ has a cycle of multiplier } w \text{ and exact period } n\}~.\]
The set $\Per^*(n,w)$ happens to be a complex hypersurface of degree $(d-1)d_n\sim (d-1)d^n$, where $d_n\geq1$ is defined by induction by $d_1=1$ and $d^n=\sum_{k|n}d_k$ for $n\geq2$. It is known that, for a fixed $w\in\C$, the sequence $d^{-n}[\Per^*(n,w)]$ converges weakly to $T_\bif$ on the moduli space $\mathcal{P}_d$. The case $|w|\leq 1$ has been established by Bassanelli and Berteloot~\cite{BB2} using an approximation formula for the Lyapunov exponent (see also~\cite{BB3}). The more delicate case $|w|>1$ has been proved recently by the first author in~\cite{DistribTbif} (see also~\cite{multipliers} for the case of quadratic polynomials with changing multipliers). 

\par Building on arithmetic methods, Favre and the first author~\cite{favregauthier} proved that postcritically finite hyperbolic parameters with $(d-1)$ distinct super-attracting cycles (resp. strictly postcritically finite parameters with given combinatorics) equidistribute towards the bifurcation measure. The proof developped in that work is only qualitative, since there exists no effective version of Yuan's arithmetic equidistribution Theorem. It also raises the question to know whether the result is of purely arithmetic nature or not. We consider the present work as a continuation of \cite{favregauthier}.

\paragraph*{Statement of the main results} Our main goal here is twofold. First, we want to establish a \emph{quantitative} equidistribution theorem for postcritically finite hyperbolic parameters. Second, we aim at giving a simpler proof than the one of~\cite{favregauthier}, relying only on pluripotential theoretic and complex analytic arguments.
To our purposes, as in the recent works \cite{favredujardin,BB2,DistribTbif,favregauthier}, we shall use the following ``orbifold'' parametrization of the moduli space $\mathcal{P}_d$.
For $(c,a)=(c_1,\ldots,c_{d-2},a)\in\C^{d-1}$, we let
\[P_{c,a}(z):=\frac{1}{d}z^d+\sum_{j=2}^{d-1}(-1)^{d-j}\sigma_{d-j}(c)\frac{z^j}{j}+a^d~, \ z\in\C~,\]
where $\sigma_{\ell}(c)$ is the monic symmetric polynomial of degree $\ell$ in $(c_1,\ldots,c_{d-2})$. Observe that the canonical projection $(c,a)\in \C^{d-1}\mapsto\{P_{c,a}\}\in\mathcal{P}_d$ is a finite branched cover of degree $d(d-1)$ and that the critical points of $P_{c,a}$ are exactly $c_0,c_1,\ldots,c_{d-2}$ with the convention that $c_0:=0$ (see Section \ref{Sec:basics} for details). For an integer $n\in \N^*$, we let $\sigma(n)$ be the sum of its divisors $\sigma(n):=\sum_{k|n} k$. The function $\sigma$ is known to be bounded from above by $C n \log \log n $ for some constant $C>0$ independent of $n$. 

Our main result may be stated as follows.
\begin{Theorem}
Let $d\geq3$. Then there exists a constant $C>0$ depending only on $d$ such that for any $(d-1)$-tuple of pairwise distinct positive integers $\un=(n_0,\ldots,n_{d-2})$ with $n_0\geq2$ and every test function $\varphi\in\mathcal{C}^2_c(\C^{d-1})$, if $\mu_{\un}$ is the probability measure equidistributed on the set of parameters in $\C^{d-1}$ for which the critical point $c_i$ is periodic of exact period $n_i$ for all $i$, we have
\[\left|\int_{\C^{d-1}} \varphi\,\mu_{\un}-\int_{\C^{d-1}}\varphi\,\mu_\bif\right|\leq C\max_{1\leq j\leq d-1}\left(\frac{\sigma(n_j)}{d^{n_j}}\right)\|\varphi\|_{\mathcal{C}^2}.\]
\label{tm:centers}
\end{Theorem}

\par The first ingredient of the proof is a (slight generalization) of a very general dynamical property established by Przytycki: if a critical point $c$ of a rational map $f$ does not lie in an attracting basin $f$, the points $c$ and $f^n(c)$ can not be to close (see Lemma~\ref{lm:Przytycki}). The idea to use Przytycki estimate in this context has been introduced by Okuyama~\cite{okuyama:distrib} (see also~\cite{okuyama:speed}) and constitutes the starting point of our work. Combined with known global properties of the family $(P_{c,a})_{(c,a)\in\C^{d-1}}$, this allows us to have precise pointwise estimates outside some specific ``bad'' hyperbolic components. The second important tool is a transversality result for critical periodic orbit relations proved in \cite{favregauthier} and relying on Epstein's transversality theory (see~\cite{epstein2}). 
 The last important ingredient we use is a $L^1$ estimate for specific solutions of the Laplacian in a bounded topological disk of an affine curve of $\C^{d-1}$, which proof crucially  relies on length-area estimates (see Theorem ~\ref{tm:L1estimatedim1}). This allows us to replace an estimate involving the \emph{diameter} of hyperbolic components with their \emph{volume}. This actually is a key step, since even in the quadratic family, estimating the diameter of hyperbolic components is a very delicate problem related to the famous hyperbolicity conjecture.
 
Nevertheless, notice that, in the context of the unicritical family, the equation $p_c^n(0)=0$ is known to have simple roots so we do not need to exclude parameters with a periodic critical point of period \emph{dividing} $n$. As a consequence, we do not use transversality statements \`a la Epstein. Hence, we will start by the proof of Theorem~\ref{tm:Mandelbrot} which is simpler and more efficient than in the general case. We may regard this as a model for the general case. 

~

Following the strategy of the proof of \cite[Theorem 3]{favregauthier}, we can deduce from Theorem \ref{tm:centers} a speed of convergence for the measure equidistributed on the (finite) set of parameters admitting $(d-1)$ distinct attracting cycles of given respective multipliers $w_0,\ldots,w_{d-2}\in\D$ and of given mutually distinct periods towards $\mu_\bif$. Let us be more precise and pick a $(d-1)$-tuple $\un:=(n_0,\ldots,n_{d-2})$ of mutually distinct positive integers and $\uw:=(w_0,\ldots,w_{d-2})\in\C^{d-1}$. When the set $\bigcap_{i=0}^{d-2}\Per^*(n_{i},w_i)$ is finite, let
\[\mu_{\un,\uw}:=\frac{1}{(d-1)!\prod_j d_{n_{j}}}\bigwedge_{j=0}^{d-2}[\Per^*(n_{j},w_j)]~.\]
Notice that $\mu_{\un,\uw}$ is a probability measure and that, when $\uw\in\D^{d-1}$, the measure $\mu_{\un,\uw}$ is exactly the measure equidistributed on the set of parameters in $\C^{d-1}$ having $(d-1)$ cycles of respective exact period $n_0,\ldots,n_{d-2}$ and multipliers $w_0,\ldots,w_{d-2}$.

The precise result we prove may be stated as follows.

\begin{Theorem}
Pick $d\geq3$. Then there exists a constant $C > 0$ such that for every $\uw=(w_0,\ldots,w_{d-2})\in\D^{d-1}$ and every $(d-1)$-tuple of pairwise distinct positive integers $\un=(n_0,\ldots,n_{d-2})$ with $n_0\geq2$, if $\mu_{\un,\uw}$ is as above, we have
\[\left|\int_{\C^{d-1}} \varphi\,\mu_{\un,\uw}-\int_{\C^{d-1}}\varphi\,\mu_\bif\right|\leq C\left(\max_{0\leq j\leq d-2}\left\{\frac{\sigma(n_j)}{d^{n_j}},\frac{-1}{d^{n_j}\log|w_j|}\right\}\right)^{\frac{1}{2}}\|\varphi\|_{\mathcal{C}^1},\]
 for any test function $\varphi\in\mathcal{C}^1_c(\C^{d-1})$.\label{cor:centers}
\end{Theorem}

\par Combining Theorem~\ref{cor:centers} with techniques from pluripotential theory (see e.g. \cite{ThelinVigny1} and \cite{dinhsibony2}), we can actually prove that for all $(d-1)$-tuples of multipliers $\uw:=(w_0,\ldots,w_{d-2})$ lying outside a pluripolar set of $\C^{d-1}$, the measure $\mu_{\un,\uw}$ equidistributed on parameters having $(d-1)$ cycles of respective multipliers $w_0,\ldots,w_{d-2}$ converges towards the bifurcation measure, as soon as the periods of the given cycles grow fast enough:
\begin{Theorem}\label{tm:principal}
Pick any sequence $\un_k=(n_{0,k},\ldots,n_{d-2,k})$ of $(d-1)$-tuples of pairwise distinct positive integers such that the series $\sum_k\max_j\{n_{j,k}^{-1}\}$ converges. Then there exists a pluripolar set of $\mathcal{E}\subset \C^{d-1} $ such that for any $\uw=(w_0,\ldots,w_{d-2})\in\C^{d-1}\setminus\mathcal{E}$, the set $\bigcap_{i=0}^{d-2}\Per^*(n_{i,k},w_i)$ is finite for any $k$ and the sequence $(\mu_{\un_k,\uw})_k$ converges to $\mu_\bif$ in the sense of measures.
\end{Theorem}

As an obvious corollary, we can deduce that if we only assume $\min_j\{n_{j,k}\}\to\infty$, then, for $\uw$ outside a pluripolar set, up to extraction $(\mu_{\un_k,\uw})_k$ converges to $\mu_\bif$ in the sense of measures. Here is another immediate and interesting consequence of Theorem~\ref{tm:principal}.

\begin{Corollary}\label{cor:neutres}
Pick any sequence $\un_k=(n_{0,k},\ldots,n_{d-2,k})$ of $(d-1)$-tuples of pairwise distinct positive integers such that the series $\sum_k \max_j(n_{j,k}^{-1})$ converges. Then, for almost any $\Theta=(\theta_0,\ldots,\theta_{d-2})\in \R^{d-1}$, if $\uw(\Theta)=(e^{2i\pi\theta_0},\ldots,e^{2i\pi\theta_{d-2}})$, the sequence $(\mu_{\un_k,\uw(\Theta)})_k$ converges to $\mu_\bif$ in the sense of measures.
\end{Corollary}

\par Notice that Bassanelli and Berteloot \cite{BB3} proved a weaker version of Corollary~\ref{cor:neutres}: they prove that the average measures $\int_{]0,1[^{d-1}}\mu_{\un_k,\uw(\Theta)}dm(\Theta)$ converge weakly to the bifurcation measure. Contrary to ours, their proof also works in any codimension.

~

\par We view these results as parametric analogues of important dynamical phenomena. Indeed, Theorem~\ref{tm:centers} is an analogue of the equidistribution of repelling periodic points of a holomorphic endomorphism $F$ of $\p^k$ towards its maximal entropy measure $\mu_F$, and Theorem~\ref{tm:principal} is an analogue of the equidistribution of preimages of a generic points, again towards the measure $\mu_F$  (see \cite{lyubich:equi,briendduval}).

\paragraph*{Perspectives}

The questions we discuss here may be addressed in a more general setting. A first natural generalization is the case when critical points can have the same period. In that case, the transversality theory \`a la Epstein fails at parameters admitting multiple critical points and we a priori have no control of the multiplicity of intersection at those parameters.

A second natural question is concerned with the case of the moduli space of degree $d$ rational maps. Even in the case of quadratic rational maps which is much better understood that the general case, important difficulties occur. The main problem comes from the fact that, contrary to the case of polynomials, the support of the bifurcation measure is \emph{not} compact in the moduli space of quadratic rational maps and that the collection of relatively compact hyperbolic components cluster at infinity (see \cite{Mod2}). 

We shall study both cases in future works.

On the other hand, in a more recent preprint (\cite{GV2}), instead of focusing on the distribution of \emph{hyperbolic} postcritically finite parameters, we study the distribution of \emph{Misiurewicz} parameters (as is also done in~\cite{favregauthier}). In this preprint, we study this problem using this time combinatorial tools developped by Kiwi~\cite{kiwi-portrait} and Dujardin-Favre~\cite{favredujardin}, enlightening slightly different, though related, phenomena.

\paragraph*{Structure of the article} Section~\ref{sec:background} is devoted to needed material. In Section~\ref{sec:premilinaires}, we establish two preliminary results: Przytycki distance estimate and the $L^1$ estimate for solutions of the Laplacian. We then present the proof of Theorem~\ref{tm:Mandelbrot} and its corollaries in Section~\ref{sec:unicritical}. The initial estimates relying on Przytycki Lemma are established in Section~\ref{sec:outside}. Section~\ref{sec:principale} is dedicated to the proof of the main Theorem~\ref{tm:centers} and Section~\ref{sec:pluripolar} to proving Theorems~\ref{cor:centers} and~\ref{tm:principal}.
Finally, in Section~\ref{sec:preimages}, we investigate other approximation phenomena. We try here to understand the distribution of maps for which the critical points are sent to some prescribed target (and not necessarily themselves). We are especially interested in a theorem of Dujardin \cite{Dujardin2012} that proves the convergence outside some pluripolar set. We give here some convergence estimates and show that in some cases, the pluripolar set can be described explicitly.

\paragraph*{Acknowledgment} The research of both authors is partially supported by the ANR project Lambda ANR-13-BS01-0002. We would like to thank Vincent Guedj for very helpful discussions concerning $L^1$ estimates for solutions of the Laplacian and Fran\c{c}ois Berteloot and Charles Favre for useful comments on preliminary versions.

\section{Background material}\label{sec:background}
In this section, we want to recall briefly background material on bifurcation currents and on classical complex analytic tools we will rely on in the whole paper.

\subsection{Holomorphic families with marked critical points}\label{sec:defbif}
Let us recall classic facts concerning holomorphic families of rational maps.

A \emph{holomorphic family} $(f_\lambda)_{\lambda\in\Lambda}$ of degree $d\geq2$ rational maps parametrized by $\Lambda$ is a holomorphic map
\[f:\Lambda\times\p^1\longrightarrow\p^1\]
such that the map $f_\lambda:= f(\lambda,\cdot):\p^1\rightarrow\p^1$ is a degree $d$ rational map, or equivalently if the map $f:\lambda\in\Lambda\longmapsto f_\lambda\in\rat_d$ is holomorphic. 

\begin{definition} We say that a holomorphic family $(f_\lambda)_{\lambda\in\Lambda}$ is \emph{with a marked critical point} if there exists a holomorphic map $c:\Lambda\rightarrow\p^1$ such that $f_\lambda'(c(\lambda))=0$ for all $\lambda\in\Lambda$.
\end{definition}
We say that a marked critical point $c$ is \emph{passive} at $\lambda_0\in\Lambda$ if there exists a neighborhood $U\subset\Lambda$ of $\lambda_0$ such that the sequence $F_n$ of holomorphic maps defined by $F_{n}(\lambda):=f_\lambda^n(c(\lambda))$ is a normal family on $U$. We say that $c$ is \emph{active} at $\lambda_0$ if it is not passive at $\lambda_0$. The \emph{activity locus} of $c$ is the set of parameters $\lambda_0\in\Lambda$ such that $c$ is active at $\lambda_0$.

Let $\omega_{\p^1}$ be the Fubini-Study form of $\p^1$ normalized so that $\int_{\p^1}\omega_{\p^1}=1$.

\begin{theorem}[Dujardin-Favre]
The sequence $d^{-n}(F_n)^*\omega_{\p^1}$ converges in the weak sense of currents to a closed positive $(1,1)$-current $T_c$ on $\Lambda$ which is supported by the activity locus of $c$. 
\end{theorem}

More precisely, there exists a locally uniformly bounded sequence of continuous function $u_n:\Lambda\times\p^1\rightarrow\R$ such that $\frac{1}{d^n}(F_n)^*\omega_{\p^1}=T_c+\frac{1}{d^n}dd^cu_n$, see e.g. \cite[Proposition-Definition 3.1]{favredujardin} or \cite{DeMarco1}. It is also known that $T_c\wedge T_c=0$ (see \cite[Theorem 6.1]{Article1}).

Let us also recall that, when a holomorphic family is with $2d-2$ marked critical points $c_1,\ldots,c_{2d-2}$, the current $T_\bif:=\sum_{i=1}^{2d-2}T_{c_i}$ is supported by the bifurcation locus in the sense of Ma\~n\'e, Sad and Sullivan (see \cite{DeMarco2}).

\begin{definition}
We define the \emph{bifurcation measure} of a family $(f_\lambda)_{\lambda\in\Lambda}$ as 
$$\mu_\bif:=\frac{1}{m!}T_\bif^m~, \ m=\dim\Lambda~.$$
\end{definition}
This measure detects, in a certain sense, the strongest bifurcations which occur in $\Lambda$. 
Finally, when $(f_\lambda)_{\lambda\in\Lambda}$ is a holomorphic family of polynomials with $d-1$ marked critical point $c_1,\ldots,c_{d-1}:\Lambda\rightarrow\C$, we let
\[g_\lambda(z):=\lim_{n\to\infty}d^{-n}\log^+|f_\lambda^n(z)|~,\]
for $\lambda\in\Lambda$ and $z\in\C$. The current $T_i$ is then given by $T_{c_i}=dd^cg_\lambda(c_i(\lambda))$ (see e.g. \cite{favredujardin}). We also can remark that it actually can be considered as equipped with $2d-2$ marked critical points letting $c_d\equiv\cdots\equiv c_{2d-2}\equiv\infty$ and that $T_i=0$ for all $d\leq i\leq 2d-2$.

\subsection{Polynomials with a specific periodic point}\label{Sec:dynatomic}
For the material of the present section, we refer to \cite[\S 4.1]{Silverman} and to \cite{BB3,BB2,Milnor3} (see also~\cite[\S 6]{favregauthier}). We follow the notations of \cite{favregauthier}.

We let $(f_\lambda)_{\lambda\in\Lambda}$ be a holomorphic family of degree $d$ polynomials parametrized by a quasi-projective variety $\Lambda$. For any $n\geq1$, the $n$-th \emph{dynatomic polynomial} of $f_\lambda$ is defined as
\[\Phi_n^*(\lambda,z):=\prod_{k|n}\left(f_\lambda^k(z)-z\right)^{\mu(n/k)}~,\]
where $\mu$ stands for the Moebius function.
This defines a polynomial map $\Phi_n:\Lambda\times \C\longrightarrow\C$ satisfying $\Phi_n(\lambda,z)=0$ if and only if
\begin{itemize}
\item either $z$ is periodic under iteration of $f_\lambda$ with $(f_\lambda^n)'(z)\neq1$ and its exact period is $n$,
\item or $z$ is periodic under iteration of $f_\lambda$ with $(f_\lambda^n)'(z)=1$ and its exact period is $k|n$ and $(f^k)'(z)$ is a primitive $n/k$-th root of unity.
\end{itemize}

When $(f_\lambda)_{\lambda\in\Lambda}$is endowed with $(d-1)$ marked critical point $c_0,\ldots,c_{d-2}:\Lambda\to\C$, we may apply this construction to those marked critical points $c_j$. We let
\[P_{n,j}(\lambda):=\Phi_n^*(\lambda,c_j(\lambda))~, \ \lambda\in\Lambda~.\]
By the above, we have 
\begin{lemma}\label{lm:dynat}
Pick $m\geq1$, $0\leq j\leq d-2$ and $\lambda\in\Lambda$. Then $P_{m,j}(\lambda)=0$ if and only if $c_j(\lambda)$ is periodic under iteration of $f_\lambda$ with exact period $m$.
\end{lemma}

We also can describe the set of parameters admitting a cycle of given period an multiplier. For $n\geq1$, set
\[p_n(\lambda,w):=\left(\mathrm{Res}_z(\Phi_n^*(\lambda,z),(f_\lambda^n)'(z)-w\right)^{1/n}~, \ (\lambda,w)\in\Lambda\times\C~.\]
This defines a polynomial $p_n:\Lambda\times\C\longrightarrow\C$. Again, we find

\begin{lemma}\label{lm:pernw}
Pick $n\geq1$ and $w\in\C$. Then $p_n(\lambda,w)=0$ if and only if one of the following occurs:
\begin{itemize}
\item if $w\neq1$, $f_\lambda$ has a cycle of exact period $n$ and multiplier $w$,
\item if $w=1$, there exists $k|n$ such that $f_\lambda$ has a cycle of exact period $k$ and multiplier $\rho$ a primitive $n/k$-root of unity.
\end{itemize}
\end{lemma}

\subsection{A specific family}

Recall that the \emph{moduli space} $\mathcal{P}_d$ of degree $d$ polynomials is the space of affine conjugacy classes of degree $d$ polynomials with $d-1$ marked critical points. We define a finite branched cover of $\C^{d-1}\to\mathcal{P}_d$ as follows. For $c=(c_1,\ldots,c_{d-2})\in\C^{d-2}$ and $a\in\C$, let
\[P_{c,a}(z):=\frac{1}{d}z^d+\sum_{j=2}^{d-1}(-1)^{d-j}\frac{\sigma_{d-j}(c)}{j}z^j+a^d~, \ z\in\C~,\]
where $\sigma_k(c)$ is the monic elementary degree $k$ symmetric polynomial in the $c_i$'s. This family is known to be a finite branched cover (see e.g. \cite[\S 5]{favredujardin}). Remark also that the critical points of $P_{c,a}$ are exactly $c_0,\ldots,c_{d-2}$, where we set $c_0:=0$, and that they depend algebraically on $(c,a)\in\C^{d-1}$.

~

We define a continuous psh function $G:\C^{d-1}\to\R_+$ by setting
$$G(c,a):=\max_{0\leq j\leq d-2}g_{c,a}(c_j)~, \ (c,a)\in\C^{d-1}.$$
It is known that the \emph{connectedness locus} 
\[\mathcal{C}_d:=\{(c,a)\in\C^{d-1}\, ; \ \J_{c,a} \ \text{is connected}\}\]
is compact and satisfies $\mathcal{C}_d=\{G=0\}$, where $\J_{c,a}:=\partial \{g_{c,a}=0\}$ is the Julia set of $P_{c,a}$ (see \cite{BH}). Moreover, the bifurcation measure, in this actual family, coincides with the Monge-Amp\`ere mass of $G$, i.e.
$$\mu_\bif=(dd^cG)^{d-1}$$
as probability measures on $\C^{d-1}$. It is also known that
$$G(c,a)=\log^+\max\{|c|,|a|\}+O(1)~,$$
where we set $|c|:=\max_{1\leq j\leq d-2}|c_j|$ , and that the function $G$ is the pluricomplex Green function of $\mathcal{C}_d$. In particular, $\mu_\bif=(dd^cG)^{d-1}$ is supported by the Shilov boundary $\partial_S\mathcal{C}_d$ of $\mathcal{C}_d$ (see \cite[\S 6]{favredujardin}). In particular, the estimates of~\cite{favredujardin} give

\begin{lemma}\label{ineq:green}
There exists a constant $C>0$ independent of $(c,a)$ such that for any $(c,a)\in\C^{d-1}$, any $z\in\C$ and any $n$,
\begin{eqnarray*}
\left|\frac{1}{d^{n}}\log^+|P_{c,a}^{n}(z)|- g_{c,a}(z)\right|\leq \frac{d^{-n+1}}{(d-1)}\cdot(\log^+\max\{|c|,|a|\}+C)~.
\end{eqnarray*}
\end{lemma}

\begin{proof}
First, let us show that 
\[\left|\frac{1}{d}\log^+|P_{c,a}(z)|-\log^+|z|\right|\leq \log^+\max\{|c|,|a|\}+C\]
for some constant $C\geq0$ depending only on $d$. As observed in the proof of \cite[Lemma 6.4]{favredujardin}, there exists $\tilde C\geq1$ such that $|P_{c,a}(z)|\leq \tilde{C}\max\{|c|,|a|,|z|\}^d$, hence
\[\frac{1}{d}\log^+|P_{c,a}(z)|\leq \log^+|z|+\log^+\max\{|c|,|a|\}+\log\tilde C\]
 for all $(c,a,z)\in\C^d$. On the other hand, if $|z|\geq d\cdot\max\{|c|,|a|,1\}$, one clearly has $|P_{c,a}(z)|\geq C'|z|^d$ for some $C'>0$ depending only on $d$ and
\[\frac{1}{d}\log^+|P_{c,a}(z)|\geq \log^+|z|-\log^+\max\{|c|,|a|\}+\log C'.\]
Finally, if $|z|\leq d\cdot\max\{|c|,|a|,1\}$, we directly find
\[\frac{1}{d}\log^+|P_{c,a}(z)|\geq 0\geq \log^+|z|-\log^+\max\{|c|,|a|\}-\log d,\]
and the conclusion follows.

Now, an immediate induction gives
\[\left|\frac{1}{d^{n+k}}\log^+|P_{c,a}^{n+k}(z)|-\frac{1}{d^n}\log^+|P_{c,a}^n(z)|\right|\leq \frac{1}{d^n}\left(\log^+\max\{|c|,|a|\}+C\right)\sum_{j=0}^{k-1}d^{-j}\]
for all $n\geq 1$ and all $k\geq1$ and the conclusion follows making $k\to\infty$.
\end{proof}


\subsection{Complex analytic tools}

We will denote in what follows $d_{\p^1}$ the classical spherical distance on $\p^1$, normalized so that $\p^1$ has diameter $1$. For a $C^1$ selfmap $f$, we denote $f^\#$ the spherical derivative of $f$:
\[\forall z\in \p^1, \   f^\#(z):= \lim_{d_{\p^1}(z,y) \to 0} \frac{d_{\p^1}(f(z),f(y))}{d_{\p^1}(z,y)}.  \]

Recall that if $A$ is an annulus and if $A$ is conformally equivalent to $A'=\{z\in\C \, ; \ r<|z|<R\}$ with $0< r<R< +\infty$, the \emph{modulus} of A is the same as the modulus of $A'$:
\[\textup{mod}(A)=\textup{mod}(A')=\frac{1}{2\pi}\log\left(\frac{R}{r}\right)~.\]
We will rely on the following classical estimate (see~\cite[Appendix]{briendduval}).

\begin{lemma}[Briend-Duval]\label{lm:BriendDuval}
For any $k\geq1$, there exists a constant $\tau_k>0$ depending only on $k$ such that for any holomorphic disks $D_1\Subset D_2\Subset\C^k$,  and any hermitian metric $\alpha$ on $\C^k$,
$$(\textup{diam}_\alpha(D_1))^2\leq \tau\cdot\frac{\textup{Area}_\alpha(D_2)}{\min(1,\textup{mod}(A))}~,$$
where $A$ is the annulus $D_2\setminus \overline{D_1}$ and areas and distances are computed with respect to $\alpha$.
\end{lemma}

We also rely on the following classical integration by part formula, which can be stated as follows (see \cite[Formula 3.1 page 144]{Demailly}).

\begin{lemma}
Let $\Omega\Subset \Omega'\subset\C^k$ be bounded open sets. Assume that $\Omega$ has smooth boundary. Let $u,v$ be psh functions on $\Omega'$ and let $T$ be a closed positive $(k-1,k-1)$-current on $\Omega'$ such that $dd^cu\wedge T$ and $dd^cv\wedge T$ are well-defined. Then
\[\int_\Omega (vdd^cu-udd^cv)\wedge T=\int_{\partial \Omega}(vd^cu-ud^cv)\wedge T~.\]\label{lm:Stokes}
\end{lemma}

\section{Preliminary results}\label{sec:premilinaires}

In this section, we establish the two main technical estimates we will rely on. The first one is of dynamical nature and follows very closely a classical result of Przytycki. The second is of more geometric nature and might be of independent interest.

\subsection{Local dynamical estimates}\label{sec:localestimates}
We shall use the following estimate in a crucial way. Though we will need it only in the case of families polynomials, we state and prove the estimate for general familes of rational maps for sake of completeness. The proof follows very closely that of \cite[Lemma 1]{Przytycki3}. The idea to use this result for proving equidistribution phenomena in parameters spaces first appeared in the recent work \cite{okuyama:distrib} of Okuyama.

\begin{lemma}\label{lm:Przytycki}
Let $(f_\lambda)_{\lambda\in\Lambda}$ be a holomorphic family of degree $d$ rational maps and let $c:\Lambda\to\p^1$ be a marked critical point of $(f_\lambda)_{\lambda \in \Lambda}$. Assume that $c(\lambda)$ does not lie persistently in a parabolic basin of $f_\lambda$, i.e. there exists $\lambda_0\in\Lambda$ such that $c(\lambda_0)$ is not attracted by a parabolic cycle of $f_{\lambda_0}$. There exists a universal constant $0<\kappa<1$ and a continuous function $M:\Lambda\to]1,+\infty[$ such that, for any $n\geq1$ and any $\lambda\in\Lambda$,  
\begin{itemize}
\item either $c(\lambda)$ lies in the immediate basin of an attracting cycle of period $p$ dividing $n$,
\item or $d_{\p^1}(f^n_\lambda(c(\lambda)),c(\lambda))\geq \kappa\cdot M(\lambda)^{-n}$.
\end{itemize}
In particular, when $c(\lambda)\in\J_\lambda$, then $d_{\p^1}(f^n_\lambda(c(\lambda)),c(\lambda))\geq \kappa\cdot M(\lambda)^{-n}$.
\end{lemma}


\begin{proof}
Notice that the functions $M_1, M :\Lambda\to]1,+\infty[$ defined by
$$M_1(\lambda):=\sup_{z\in\p^1}f_\lambda^\#(z)\in]1,+\infty[$$
and 
$$M(\lambda):=M_1(\lambda)^2$$
are continuous on $\Lambda$. Moreover, the map $f_\lambda:\p^1\to\p^1$ is Lipschitz with constant $M_1(\lambda)$ with respect to $d_{\p^1}$, i.e.
$$d_{\p^1}(f_\lambda(z),f_\lambda(w))\leq M_1(\lambda)d_{\p^1}(z,w)~,$$
for any $z,w\in\p^1$ and any $\lambda\in\Lambda$.  We rely on the following.

\begin{claim}
There exists a constant $\kappa>0$ such that for every $\lambda\in\Lambda$ and any $n\geq1$ if $f_\lambda^n(c(\lambda))\neq c(\lambda)$ and $d_{\p^1}(c(\lambda),\J_\lambda)< \kappa M(\lambda)^{-n}$, then
\[d_{\p^1}(f^n_\lambda(c(\lambda)),c(\lambda))\geq \kappa M(\lambda)^{-n}~.\]
\end{claim}
Fix $\lambda\in\Lambda$ and $n\geq1$. We now assume all along the proof that $c(\lambda)$ does not belong to an attracting basin of period $p|n$. In particular, $f_\lambda^n(c(\lambda))\neq c(\lambda)$.  When $c(\lambda)\in \J_\lambda$, the above Claim implies
\[d_{\p^1}(f^n_\lambda(c(\lambda)),c(\lambda))\geq\kappa\cdot M(\lambda)^{-n}~,\]
as required.
Assume now $c(\lambda)\notin \J_\lambda$. There are two distinct cases to treat. First, assume $c(\lambda)$ lies in a parabolic basin of $f_{\lambda}$. Since $c(\lambda)$ is not persistently in such a component, the critical point is active at $\lambda$, hence there exists $\lambda_k\to\lambda$ with $c(\lambda_k)\in\J_{\lambda_k}$ by Montel Theorem. By continuity of the function $\lambda\mapsto d_{\p^1}(f^n_\lambda(c(\lambda)),c(\lambda))-\kappa\cdot M(\lambda)^{-n}$, we find
\[d_{\p^1}(f^n_\lambda(c(\lambda)),c(\lambda))-\kappa\cdot M(\lambda)^{-n}=\lim_{k\to\infty}\left(d_{\p^1}(f^n_{\lambda_k}(c(\lambda_k)),c(\lambda_k))-\kappa\cdot M(\lambda_k)^{-n}\right)\geq0\]
in that case, i.e. $d_{\p^1}(f^n_\lambda(c(\lambda)),c(\lambda))\geq\kappa\cdot M(\lambda)^{-n}$.

~

In any other case, by assumption, $f_\lambda^n(c(\lambda))$ lies in a different Fatou component of $f_\lambda$ than $c(\lambda)$. Then, either $d_{\p^1}(c(\lambda),\J_\lambda)<\kappa\cdot M(\lambda)^{-n}$ which, by the Claim, implies
\[d_{\p^1}(f^n_\lambda(c(\lambda)),c(\lambda))\geq \kappa M(\lambda)^{-n}~,\]
or $d_{\p^1}(c(\lambda),\J_\lambda)\geq\kappa\cdot M(\lambda)^{-n}$ and we have
$$d_{\p^1}(f^n_\lambda(c(\lambda)),c(\lambda))\geq d_{\p^1}(c(\lambda),\J_\lambda)\geq\kappa\cdot M(\lambda)^{-n}$$
and the proof is complete.
\end{proof}


We are left with proving the Claim.
\begin{proof}[Proof of the Claim] In what follows, we let $\B(z,r)$ (resp. $\B^e(z,r)$) denote the spherical (resp. euclidean) ball of center $z$ and radius $r$. Recall that the group of (holomorphic) isometries for the spherical metric acts transitivelly on $\p^1$. Pick some parameter $\lambda$ and let $A$ and $B$ be such isometries. Then for $r>0$:
	\begin{equation*}
	\textup{diam}\,f_\lambda(\B(c(\lambda),r))= 	\textup{diam}\, A\circ f_\lambda \circ B ( B^{-1}(\B(c(\lambda),r)))= \textup{diam}\, A\circ f_\lambda \circ B ( (\B( B^{-1}(c(\lambda)),r))).
	\end{equation*}
Choosing $A$ and $B$ such that $A(f_\lambda(c(\lambda)))=0$ and $B^{-1}(c(\lambda))=0$ gives
	\begin{equation*}
	\textup{diam}\,f_\lambda(\B(c(\lambda),r))= \textup{diam}\, A\circ f_\lambda \circ B ( (\B(0,r))).
	\end{equation*}	
Let $g$ be the rational map of $\p^1$ defined by $g:= A\circ f_\lambda \circ B$ so that $g(0)=0$. As $A$ and $B$ are isometries, we have $g^\#(z)= f_\lambda^\#(B(z))$ so that $\sup_{z\in\p^1}g^\#(z)= \sup_{z\in\p^1}f_\lambda^\#(z)$. Hence $g$ is $M_1(\lambda)$-Lipschtiz. Set now $r_0 :=\frac{1}{2M_1(\lambda)} (<1/2)$, then:
	\begin{equation*}
\forall z \in \B(0,r_0), \ d_{\p^1}(g(z),g(0))= d_{\p^1}(g(z),0)\leq \frac{1}{2} ~,
	\end{equation*}	
so that $ g(\B(0,r_0)) \subset \B(0,1/2)$. As the spherical metric and the euclidean metric are comparable on $\B(0,1/2)$, there exists $c >1$ such that for all $r\leq 1/2$,  
$$\B^e\left(0,r\right)\subset \B(0,r)\subset \B^e(0,cr).$$ 
In particular,
$$ \B^e\left(0,\frac{1}{2M_1(\lambda)}\right)\subset \B\left(0,\frac{1}{2M_1(\lambda)}\right).$$ 
so that :
\begin{equation*} 
g\left(\B^e\left(0,\frac{1}{2M_1(\lambda)}\right)\right)\subset \B(0,1/2)\subset \B^e(0,c/2).
\end{equation*}
Whence $g:\B^e\left(0,\frac{1}{2M_1(\lambda)}\right)\to \B^e(0,c/2)$ is holomorphic. Applying Cauchy's inequality for $g'$ we find
$$\|g''\|_{\infty, \B^e(0,\frac{1}{4M_1(\lambda)})} \leq \frac{\|g'\|_{\infty,\B^e(0,\frac{1}{2M_1(\lambda)})}}{(4M_1(\lambda))^{-1}}\leq c' M_1(\lambda)^2=c' M(\lambda) $$ 
	where $c'$ does not depend on $\lambda$ and we used that $|g'|\leq c g^\#$ on $\B^e(0,\frac{1}{2M_1(\lambda)})$. Since $g'(0)=0$, up to replacing $c'$ by $c^2c'$, we deduce that
	$\|g'\|_{\infty, \B^e(0,cr)} \leq c' M(\lambda) r,$ 
	and so $g(\B^e(0,cr))\subset \B^e(0, c' M(\lambda) r^2)$, for all $r\leq \frac{1}{4cM_1(\lambda)}$. Summing up, we have:
	\begin{equation*}
	g(\B(0, r))\subset \B(0, c' M(\lambda) r^2 ), \text{ for all } r \leq \frac{1}{4cM_1(\lambda)}.
	\end{equation*}
Going back to $f_\lambda$ and taking the diameter we deduce (using that $M_1(\lambda)\leq M(\lambda)$):
\begin{equation}\label{mean_value}
\forall r \leq \frac{1}{4cM(\lambda)}, \textup{diam}\,f_\lambda(\B(c(\lambda),r))\leq c' M(\lambda)^2 r^2,
\end{equation}
	where $c$ and $c'$ are constants that depend neither on $\lambda$ nor on $r$.
	
We now follow the proof of \cite[Lemma 1]{Przytycki3}.  Take $\kappa$ so that
$ \kappa \leq \frac{1}{4c'} $ and $\kappa \leq \frac{1}{8c}$ ($c$ and $c'$ are the constants in \eqref{mean_value}). Assume, by contradiction, that $d_{\p^1}(c(\lambda),\J_\lambda)< \kappa M(\lambda)^{-n}$, $d_{\p^1}(f^n_\lambda(c(\lambda)),c(\lambda))<\kappa M(\lambda)^{-n}$ and $f_\lambda^n(c(\lambda))\neq c(\lambda)$.

 Choose
$$\epsilon:=\max\{d_{\p^1}(c(\lambda),\J_\lambda),d_{\p^1}(c(\lambda),f^n_\lambda(c(\lambda)))\}>0~.$$
Then $\overline{\B}(c(\lambda),\epsilon)\cap\J_\lambda\neq\emptyset$. 
Since $f_\lambda$ is $M_1(\lambda)$-Lipschitz with respect to $d_{\p^1}$, we get for $j\geq 2$:
$$ \textup{diam}\,f^j_\lambda(\B(c(\lambda),2\epsilon))\leq  M_1(\lambda)\textup{diam}\,f^{j-1}_\lambda(\B(c(\lambda),2\epsilon)).$$
By our choice of $\kappa$, we have that $2\epsilon \leq \frac{1}{4cM(\lambda)}$, so combining with \eqref{mean_value} gives:
$$ \textup{diam}\,f^n_\lambda(\B(c(\lambda),2\epsilon))\leq 4c' M_1(\lambda)^{n+1}\epsilon^2\leq  4c' M(\lambda)^{n}\epsilon^2 <  \epsilon~.$$
In particular, $f_\lambda^n(\B(c(\lambda),2\epsilon))\subset \B(c(\lambda),2\epsilon)$ and $f_\lambda^{kn}(\B(c(\lambda),2\epsilon))\subset \B(c(\lambda),2\epsilon)$ follows for any $k\geq1$ by an immediate induction. By Montel's Theorem, the sequence $(f^{kn}_\lambda)_{k\geq1}$ is a normal family on $\B(c(\lambda),2\epsilon)$, which is a contradiction, since $\J_\lambda\cap\B(c(\lambda),2\epsilon)\neq\emptyset$.
\end{proof}

\subsection{$L^1$-estimate for local solutions of the Laplacian on affine curves}\label{Sec:apriori}

Our aim, in the present, is to give the following  $L^1$-estimate for solutions of the Laplacian in disks of algebraic curves. We let $\beta$ be the standard hermitian metric on $\C^k$. Precisely, we want to prove the following.

\begin{theorem}\label{tm:L1estimatedim1}
Pick $k\geq1$. There exists a constant $C > 0$ depending only on $k$ such that for every affine algebraic curve $S\subset\C^k$, every simply connected bounded domain of $\Omega$ in $S$ satisfying $\Omega\cap S_{\textup{sing}}=\emptyset$, every $z_0\in\Omega$ and every $f\in\mathcal{C}(\overline{\Omega},\C)$ holomorphic on $\Omega$ satisfying $dd^c\log|f|=\delta_{z_0}$ on $\Omega$, we have
\[\left\|\log|f|\right\|_{L^1(\Omega,\beta)} \leq C\max\left\{\left\|\log|f|\right\|_{L^\infty(\partial\Omega)},1\right\}\textup{Area}_\beta(\Omega)~.\]
\end{theorem}

\begin{proof}
Set $u:=\log|f|$. First, if $f$ vanishes somewhere on $\partial\Omega$, then the estimate is trivial since $\left\|u\right\|_{L^\infty(\partial\Omega)}=+\infty$ then. We thus assume $f(\zeta)\neq0$ for all $\zeta\in\partial\Omega$.

 Let $h:\Omega\to\D$ be a biholomorphic map with $h(z_0)=0$. It is clear that $|h|$ extends continuously to $\partial\Omega$ with $|h|=1$ on $\partial\Omega$. We set $\chi:=\log|h|$ and $K:=\max\left\{\|u\|_{L^\infty(\partial\Omega)},1\right\}$. The function $\chi$ is subharmonic on $\Omega$, satisfies $dd^c\chi=\delta_{z_0}$ and $\chi\leq0$. The functions $h/f$ and $f/h$ are holomorphic on $\Omega$ and satisfy
\[\left|\frac{f}{h}\right|\leq \exp(K) \ \text{ and } \ \left|\frac{h}{f}\right|\leq \exp(K) \ \text{ both on }\partial\Omega~.\]
Using twice the maximum principle, we get $\chi-K\leq u\leq K+\chi\leq K-\chi$ on $\Omega$. Define
$$\Omega':=\left\{z\in\Omega\, ; \ \chi(z)<-3K\right\} \ \text{ and } \ \Omega'':=\left\{z\in\Omega\, ;  \chi(z)<-\frac{5}{2}K\right\}~.$$
We can decompose $\|u\|_{L^1(\Omega)}$ as follows
\begin{eqnarray*}
\|u\|_{L^1(\Omega)} & = & \int_{\Omega}|u|\,\beta\leq  \int_{\Omega}(K-\chi)\,\beta\leq K\textup{Area}_\beta(\Omega)-\int_{\Omega}\chi\,\beta\\
& \leq & K\textup{Area}_\beta(\Omega)-\int_{\Omega\setminus\Omega'}\chi\,\beta-\int_{\Omega'}\chi\,\beta\\
& \leq & 4K\textup{Area}_\beta(\Omega)-\int_{\Omega'}\chi\,\beta~.
\end{eqnarray*}
We are thus left with estimating $\int_{\Omega'}\chi\,\beta$ from below. Writing $v(z):=\|z-z_0\|^2$, where $\|\cdot\|$ is the euclidean norm of $\C^k$, we have $\beta=dd^cv$ and by Stokes (see Lemma~\ref{lm:Stokes}):
\begin{eqnarray*}
\int_{\Omega'}\chi\,\beta & = & \int_{\Omega'}v\,dd^c\log|h|-\int_{\partial \Omega'}v\,d^c\log|h|+\int_{\partial \Omega'}\log|h|\,d^cv~.
\end{eqnarray*}
As $dd^c\log|h|=\delta_{z_0}$, $d^c \log|h|^2= d^c |h|^2 / |h|^2$  and by definition of $\Omega'$,  we find
\begin{eqnarray*}
\int_{\Omega'}\chi\,\beta & = &-3K\int_{\partial \Omega'}\,d^cv-\int_{\partial \Omega'}v\,d^c\log|h|\\
& = & -3K\text{Area}_\beta(\Omega')-\frac{1}{2}\exp(6K)\int_{\partial \Omega'}v\,d^c|h|^2~.
\end{eqnarray*}
Applying again Stokes yields
\begin{eqnarray*}
\int_{\partial \Omega'}v\,d^c|h|^2 & = & \int_{\Omega'}v\,dd^c|h|^2-\int_{\Omega'}|h|^2\,\beta +\int_{\partial \Omega'}|h|^2\,d^cv\\
&\leq & \int_{\Omega'}v\,dd^c|h|^2   + \exp(-6K)\textup{Area}_\beta(\Omega')
\end{eqnarray*}
since $|h|\leq 1$ and $v\geq 0$.

All the above estimates summarize as follows
\begin{eqnarray}
0\geq\int_{\Omega'}\chi\,\beta & \geq & -4K\cdot\textup{Area}_\beta(\Omega')-\frac{1}{2}\exp(6K)\int_{\Omega'}vdd^c|h|^2~,\label{eq:ineg}
\end{eqnarray}
since $K\geq1$ and $|h|\geq0$. Since $v(z)=\|z-z_0\|^2$, we can bound it in $\Omega'$ from above by $(\textup{diam}_\beta(\Omega'))^2$. We now apply Lemma \ref{lm:BriendDuval} to $\Omega'\Subset\Omega''$ and find
\begin{eqnarray}
\int_{\Omega'}vdd^c|h|^2 & \leq & (\textup{diam}_\beta(\Omega'))^2\cdot\int_{\Omega'}dd^c|h|^2
=(\textup{diam}_\beta(\Omega'))^2\cdot\int_{\D(0,e^{-3K})}\beta\label{remarkreferee2}\\
& \leq & (\textup{diam}_\beta(\Omega'))^2\cdot \pi\exp(-6K)\nonumber\\
& \leq & \pi\exp(-6K)\tau\cdot\frac{\textup{Area}_\beta(\Omega'')}{\min\left\{1,\textup{mod}(\Omega''\setminus \overline{\Omega'})\right\}}\nonumber~.
\end{eqnarray}
Since $h$ is a biholomorphism and since $K\geq1$, we have
\[\textup{mod}(\Omega''\setminus \overline{\Omega'})=\textup{mod}(\D(0,e^{-5K/2})\setminus\overline{\D(0,e^{-3K})})= \frac{K}{2\pi}\geq\frac{1}{2\pi}~.\]
Taking $C:=8+\pi^2\tau$ ($K\geq1$) ends the proof.
\end{proof}

\begin{remark}\rm
Although it seems that all the analysis can be made directly on $\D$, in \eqref{remarkreferee2} we need to work on $\Omega'$ to bound $v$ by the euclidean diameter of $\Omega'$.
\end{remark}

\section{Speed of convergence in the unicritical family}\label{sec:unicritical}

This section is devoted to the proof of Theorem \ref{tm:Mandelbrot}. The method we use can be seen as a toy-model for the proof of Theorem \ref{tm:centers}. 
Our idea consists in giving estimates in $L^1_{\textup{loc}}$ for the sequence $c\mapsto d^{-n+1} \log |p_c^n(0)| - g_{\Mand_d}$ of DSH functions (difference of subharmonic functions) and then to deduce Theorem \ref{tm:Mandelbrot} from these estimates.

\subsection{Preliminaries}

Let $d\geq 2$. Recall that the \emph{Multibrot set} $\Mand_d$ is defined by
\[\Mand_d:=\{c\in\C \, ; \ \J_c \text{ is connected}\}~.\]
  Observe that $0$ on $\C$ is the unique marked
critical point of $p_c$ (other than $\infty$). According to the notation introduced in section \ref{sec:defbif}, the Green function should be denoted by $g_0$, we shall instead use the notations $g_c$ which is more classical in that setting. In particular, 
$\Mand_d$ coincides with the set $\{c\in\C\, ; \ g_c(0)=0\}$. The bifurcation locus of the family $(p_c)_{c\in\C}$ is known to coincide with the boundary $\partial\Mand_d$ of the Multibrot set. In this family, the bifurcation measure is given by
\[dd^cg_c(0)=\frac{1}{d}dd^cg_c(c)=\frac{1}{d}dd^cg_{\Mand_d}\]
where $g_{\Mand_d}$ is the Green function of the compact set $\Mand_d$ (see e.g. \cite{Steinmetz}). In particular, $dd^cg_{\Mand_d}=\mu_{\Mand_d}$ is the equilibrium measure (and the harmonic measure) of $\Mand_d$.

~

Let us first prove the following. 

\begin{lemma}
There exists $C_1>0$ depending only on $d$ such that
\[\left|d^{-n+1}\log|p_c^n(0)|-g_{\Mand_d}(c)\right|\leq\frac{n}{d^{n-1}}C_1 ~,\]
 for any $c\in\C\setminus\Mand_d$ and any $n\geq1$.\label{lm:outside}
\end{lemma}

\begin{proof}
Let $n\geq1$ and set $u_n(c):=d^{-n+1}\log|p_c^n(0)|$, $c\in\C$. A classical and easy computation shows that, for any $c\in\Mand_d$, one has $|p_c^n(0)|\leq2$(e.g. \cite{Orsay1}).
Let $M_1:=\sup_{c\in\partial \Mand_d}M(c)$, where $M(c)$ is the constant given by Lemma \ref{lm:Przytycki}. Then, for all $c\in\partial\Mand_d$,
$$-\frac{1}{d^{n-1}}\left(n\log M_1-\log\kappa\right)\leq u_n(c)-g_{\Mand_d}(c)\leq \frac{1}{d^{n-1}}\log2~,$$
since $g_{\Mand_d}(c)=0$. Setting $C_1:=2\max\{\kappa^{-1},M_1,2\}$, we get for any $c\in\partial\Mand_d$
\begin{eqnarray}
-\frac{n}{d^{n-1}}\log C_1\leq u_n(c)-g_{\Mand_d}(c)\leq \frac{n}{d^{n-1}}\log C_1~.\label{eq:endehorsdeMandelbrot}
\end{eqnarray}
Now, as $c\to\infty$, by definition of $p_c$, one has 
\[\lim_{c\to\infty}\left(u_n(c)-\log|c|\right)=\lim_{c\to\infty}\left(d^{-(n-1)}\log\left|c^{d^{n-1}}\right|-\log|c|\right)=0~.\]
As $g_{\Mand_d}(c)=\log|c|+o(1)$ as $c\to\infty$ (by definition of the Green function), the function $u_n-g_{\Mand_d}$ extends continuously on the xhole  $\p^1\setminus\Mand_d$. Moreover, $h_n:=u_n-g_{\Mand_d}$ is harmonic on $\C\setminus\Mand_d$. 

The function $h$ is thus subharmonic and continuous on $\p^1\setminus\Mand_d$. Moreover, since $h_n$ is harmonic on $\C\setminus\Mand_d$, one has $\textup{supp}(\Delta h_n)\cap(\p^1\setminus\Mand_d)\subset\{\infty\}$. As $h_n$ is continuous at $\infty$, it can not have a Dirac mass here, hence $\Delta h_n=0$ on $\p^1\setminus\Mand_d$, i.e. $h_n$ is harmonic on $\p^1\setminus\Mand_d$. By the maximum principle, applied successively to $h_n$ and $-h_n$, \eqref{eq:endehorsdeMandelbrot} gives the wanted estimate.
\end{proof}

\subsection{In the Multibrot set}
For every $n\geq1$, we denote $v_n:=\log|p_c^n(0)|$ (hence $u_n=d^{-n+1}v_n$). We will prove that:
\begin{theorem}\label{lm:L1loc}
There exists $C>0$ depending only on $d$ such that for any $n\geq1$,
\[\|v_n \|_{L^1(\Mand_d)}\leq C\cdot n\cdot \textup{Area}(\Mand_d)~.\]
\end{theorem}
The next lemma will be use in the proof of Theorem~\ref{lm:L1loc}. Though it is classical, we include a proof for the sake of completeness.
\begin{lemma}
The polynomial $p_c^n(0)\in\C[c]$ has only simple roots in $\C$.\label{lm:multi}
\end{lemma}
\begin{proof}
Let $Q_n(c):=p_c^n(0)\in\mathbb{Z}[c]$. By definition, $Q_n(c)=Q_{n-1}(c)^d+c$, hence $Q_n'=dQ_{n-1}'Q_{n-1}^{d-1}+1$. In particular, the discriminant $\Delta(Q_n)=\textup{Res}(Q_n,Q_n')$ is the determinant of a matrix with entries divisible by $d$ below the diagonal and equal to $1$ on the diagonal. Hence reducing modulo $d$, it is upper triangular with 1 on the diagonal, i.e. $\Delta(Q_n)\equiv1 \ (\textrm{mod }d)$. In particular, if $Q_n(c_0)=0$, then $Q_n'(c_0)\neq0$ and the point $c_0$ is a simple root of $Q_n$.
\end{proof}
Let $n\geq1$. We denote by $\mathcal{H}_n$ the union of connected components $\Omega$ of the interior of $\Mand_d$ such that $\Per(n)\cap\Omega\neq\emptyset$. 
For $\Omega$ such a connected component, recall that $\Omega$ is simply connected and there exists only one $c_\Omega \in \Omega$ for which  $p_{c_\Omega}^n(0)=0$ (see~\cite{Orsay2}).

\begin{remark} \rm
This can be seen using the multiplier map $\rho:\Omega \to \D $ which is a branched cover of degree $d-1$ which is totally ramified at $\{c_\Omega\}:=\rho^{-1}\{0\}=\Per(n)\cap\Omega$. 
\end{remark}

The proof of Theorem~\ref{lm:L1loc} is an application of Theorem~\ref{tm:L1estimatedim1}.
\begin{proof}
We decompose $\|v_n\|_{L^1(\Mand_d)}$ as follows
\begin{eqnarray*}
\|v_n\|_{L^1(\Mand_d)}=\int_{\Mand_d}|v_n|\,\beta=\int_{\Mand_d\setminus\mathcal{H}_n}|v_n|\,\beta+\int_{\mathcal{H}_n} |v_n|\,\beta~.
\end{eqnarray*}
Let $U$ be a connected component of the interior of $\Mand_d\setminus\mathcal{H}_n$. By Lemma~\ref{lm:outside} (recall that $g_{\Mand_d}(c)=0$ in $\Mand_d$), we have that $|v_n|\leq  n C_1$ on $\partial U$.
As $v_n$ is harmonic on $U$, by the maximum principle, the estimate extends to $U$:
 \[ |v_n|\leq  nC_1 \ \text{ on } U.\]
 Hence:
 $$\int_{\Mand_d\setminus\mathcal{H}_n}|v_n|\,\beta \leq nC_1 \textup{Area}_\beta(\Mand_d\setminus\mathcal{H}_n) ~.$$
 Now, let $\Omega$ be a connected component of $\mathcal{H}_n$. Then, by Lemma~\ref{lm:multi}, the function $v_n$
satisfies $dd^c v_n =\delta_{c_\Omega}$ where $c_\Omega$ is the center of the component $\Omega$. Hence using Theorem~\ref{tm:L1estimatedim1} and Lemma~\ref{lm:outside} gives (take $\log |f|=v_n$):
\[\left\|v_n\right\|_{L^1(\Omega)} \leq C\max\left\{\left\|v_n\right\|_{L^\infty(\partial\Omega)},1\right\}\textup{Area}_\beta(\Omega) \leq Cn C_1\textup{Area}_\beta(\Omega)~,\]
where $C$ is a universal constant that depends only on $d$. Summing on all the connected components $\Omega$ of $\mathcal{H}_n$ gives:
\[\int_{\mathcal{H}_n}|v_n| \,\beta\leq Cn C_1\textup{Area}_\beta(\mathcal{H}_n)~.\]
Summing over $\mathcal{H}_n$ and $\Mand_d\setminus\mathcal{H}_n$ ends the proof.  
\end{proof}

\subsection{Proof of Theorem \ref{tm:Mandelbrot}}

Let us now explain how to deduce Theorem \ref{tm:Mandelbrot} from Theorem \ref{lm:L1loc}. 

\begin{proof}[Proof of Theorem \ref{tm:Mandelbrot}]
 Let $n\geq 1$. Recall that $u_n(c)=d^{-n+1}\log|p_c^n(0)|=d^{-n+1} v_n$.
 Let $\varphi\in\mathcal{C}^2_c(\C)$. Then, by Stokes formula:
\begin{eqnarray*}
\frac{1}{d^{n-1}}\sum_{c\in\Per(n)}\varphi(c)-\int_\C\varphi\, \mu_{\Mand_d} & = & \int_\C\varphi\ dd^cu_n-\int_\C\varphi\ dd^cg_{\Mand_d}=\int_\C(u_n-g_{\Mand_d})\ dd^c\varphi~.
\end{eqnarray*}
We cut the integral into two parts
\begin{eqnarray*}
\frac{1}{d^{n-1}}\sum_{c\in\Per(n)}\varphi(c)-\int_\C\varphi\, \mu_{\Mand_d} & = & \int_{\Mand_d} (u_n-g_{\Mand_d})dd^c\varphi+\int_{\C\setminus\Mand_d}(u_n-g_{\Mand_d}) dd^c\varphi~.
\end{eqnarray*}
Now, as $\varphi$ is $\mathcal{C}^2$, we have (up to a constant that depends on the choice of the $\mathcal{C}^2$-norm) that $\|\varphi\|_{\mathcal{C}^2}\omega_{\p^1} \pm dd^c \varphi \geq 0$, where $\omega_{\p^1}$ is the Fubini-Study form on $\p^1$ (normalized so that $\omega_{\p^1}(\p^1)=1$), hence we can write 
\begin{center}
$\displaystyle\left|\frac{1}{d^{n-1}}\sum_{c\in\Per(n)}\varphi(c)-\int_\C\varphi\, \mu_{\Mand_d}\right| \leq \|\varphi\|_{\mathcal{C}^2}\left(\int_{\Mand_d} |u_n-g_{\Mand_d}|\omega_{\p^1}+\int_{\C\setminus\Mand_d}|u_n-g_{\Mand_d}| \omega_{\p^1}\right).$
\end{center}
As $\omega_{\p^1}\leq\beta$, where $\beta$ is the standard hermitian metric on $\C$,
\begin{center}
$\displaystyle\left|\frac{1}{d^{n-1}}\sum_{c\in\Per(n)}\varphi(c)-\int_\C\varphi\, \mu_{\Mand_d}\right|  \leq \|\varphi\|_{\mathcal{C}^2}\left(\int_{\Mand_d} |u_n-g_{\Mand_d}|\beta+\int_{\C\setminus\Mand_d}|u_n-g_{\Mand_d}| \omega_{\p^1}\right).$
\end{center} 
By Lemma ~\ref{lm:outside}, we have the bound:
$$ \int_{\C\setminus\Mand_d}|u_n-g_{\Mand_d}| \omega_{\p^1} \leq  \frac{n}{d^{n-1}} C_1 \int_{\p^1} \omega_{\p^1}=\frac{n}{d^{n-1}} C_1 . $$
Now, by Theorem~\ref{lm:L1loc}, there exists $C>0$ depending only on $d$ such that for any $n\geq1$,
\[\|v_n \|_{L^1(\Mand_d)}\leq C\cdot n\cdot \textup{Area}(\Mand_d)~.\]
Hence
\begin{eqnarray*}
\int_{\Mand_d} |u_n-g_{\Mand_d}|\beta & = & d^{-n+1}\int_{\Mand_d} |v_n|\beta\\
& \leq & \textup{Area}(\Mand_d)\cdot\frac{C' n }{d^n},
\end{eqnarray*}
where $C'$ depends only on $d$. Combining the estimates gives that there exists a constant $C$ that depends only on $d$ such that:
\begin{eqnarray*}
\left|\frac{1}{d^{n-1}}\sum_{c\in\Per(n)}\varphi(c)-\int_\C\varphi\, \mu_{\Mand_d} \right| & \leq & \frac{C n }{d^n}\|\varphi\|_{\mathcal{C}^2}~.
\end{eqnarray*}
This ends the proof. 
\end{proof}

\begin{remark} \rm
\begin{enumerate}
\item Observe that in fact, one can replace in Theorem \ref{tm:Mandelbrot} the norm $\|\varphi\|_{\mathcal{C}^2}$ by the $L^\infty$ norm of $dd^c \varphi$.
\item On the other hand, as the measures 
$$\frac{1}{d^{n-1}}\sum_{c\in\Per(n)} \delta_c \ \mathrm{and} \  \mu_{\Mand_d}$$
all have supports in $\Mand_d$, if $\theta$ denotes a cut-off function equal to $1$ in a neighborhood of $\Mand_d$ then for all $\varphi$:
$$\frac{1}{d^{n-1}}\sum_{c\in\Per(n)}\varphi(c)-\int_\C\varphi\, \mu_{\Mand_d}= \frac{1}{d^{n-1}}\sum_{c\in\Per(n)}\theta(c)\varphi(c)-\int_\C \theta \varphi\, \mu_{\Mand_d}.$$
Then, one easily gets that $\|\theta \varphi\|_{\mathcal{C}^2} \leq A\|\varphi\|_{\mathcal{C}^2(K)}$ where $K=\mathrm{supp}( \theta)$ and $A$ is a constant that depends only on $\theta$. Then we have the estimate, for all $\varphi \in \mathcal{C}^2(\C)$:
$$\left|\frac{1}{d^{n-1}}\sum_{c\in\Per(n)}\varphi(c)-\int_\C\varphi\, \mu_{\Mand_d} \right|  \leq  CA\frac{ n }{d^n}\|\varphi\|_{\mathcal{C}^2(K)}~. $$
\end{enumerate}
\end{remark}

\subsection{An application in the spirit of Theorem~\ref{cor:centers}}
We here want to extend the $\mathcal{C}^1$-estimate to non-postcritically finite but hyperbolic parameters, i.e. parameters $c\in\Mand_d$ for which $p_c$ has a cycle of period $k|n$ and multiplier $w^{k/n}$ where $w\in\D^*$ has been fixed. Let us be more precise:
For any $n\geq1$ and any $w\in\C$, let $R_n(c,w):=\mathrm{Res}_z\left(p_c^n(z)-z,(p_c^n)'(z)=w\right)$ and
\begin{eqnarray*}
\Per(n,w) & := & \{c\in\C \, ; \ R_n(c,w)=0\}~.
\end{eqnarray*}
By the above section \S~\ref{Sec:dynatomic}, we have
\begin{eqnarray*}
\Per(n,w) & = & \{c\in\C \, ; \ \exists z\in\C~, \ p_c^n(z)=z \ \text{and} \ (p_c^n)'(z)=w\}\\
& = & \bigcup_{k|n}\Per^*(k,w^{n/k})~.
\end{eqnarray*}
Our precise statement is the following.

\begin{theorem}
For any integer $n\geq1$, the set $\Per(n,w)$ is a finite set of cardinal $d^{n-1}$ (resp. $(d-1)d^{n-1}$) if $w=0$ (resp. $w\in\D^*$). Moreover, there exists a constant $C>0$ depending only on $d$ such that
$$\left|\frac{1}{d^{n-1}}\sum_{c\in\Per(n,w)}\frac{\varphi(c)}{(d-1)}-\int\varphi\, \mu_{\Mand_d}\right|\leq C\left(\frac{n}{d^n}\right)^{\frac{1}{2}}\max\left\{1,\frac{1}{\log(|w|^{-1})}\right\}^{\frac{1}{2}}\|\varphi\|_{\mathcal{C}^1},$$
 for any $w\in\D^*$, any $\varphi\in\mathcal{C}^1_c(\C)$ and any $n\geq1$.
\label{cor:Mandelbrot}
\end{theorem}

The proof of Theorem~\ref{cor:Mandelbrot} uses Lemma~\ref{lm:BriendDuval}, i.e. the length-area estimates of Briend and Duval~\cite{briendduval} in a crucial way. Notice that we may instead use Koebe Distortion estimates, but we again want to present the proof of Theorem~\ref{cor:Mandelbrot} as a toy model for that of Theorem \ref{cor:centers}. 

Notice that, though this estimate looks weaker than the one obtained in Theorem~\ref{cor:centers}, it is actually more general. Indeed, the set $\Per(n,w)$ as we defined it here is the set of parameters $c$ for which there exists a cycle of period \emph{dividing} $n$ and multiplier \emph{a root of} $w$. Hence, this is a much bigger set than the one involved in Theorem~\ref{cor:centers}.

 In general, the set $\Per(n,w)$ with $w\in\C$ is finite and has cardinal at most $(d-1)d^{n-1}$, see e.g.~\cite{Silverman}. 

\begin{proof} Let us first prove that $\Per(n,w)$ is finite and determine its cardinal. If $c\in\Per(n,w)$, then $p_c$ has an attracting cycle of exact period $k|n$ and multiplier $w^{k/n}\in\D$. But the set of parameters $c$ admitting a $k$-cycle of multiplier $t\in\D$ is finite and its cardinal is $d_{k-1}$ if $t=0$ and $(d-1)d_{k-1}$ otherwise. In particular, $\Per(n,w)$ is finite and has cardinal $\sum_{k|n}d_{k-1}=d^{n-1}$ if $w=0$ and $\sum_{k|n}(d-1)d_{k-1}=(d-1)d^{n-1}$ otherwise (recall that $d_n$ was defined in the introduction).

Pick $\varphi\in\mathcal{C}^1_c(\C)$ and let
\[\mu_n:=\frac{1}{d^n}\sum_{c\in\Per(n)}\delta_c \ \and \ \mu_{n,w}:=\frac{1}{(d-1)d^n}\sum_{c\in\Per(n,w)}\delta_c~.\]
Recall that we want to estimate $|\langle\mu_{n,w},\varphi\rangle-\langle\mu_{\Mand_d},\varphi\rangle|$. Let us remark that, by a classical interpolation argument, Theorem~\ref{tm:Mandelbrot} gives
\[|\langle\mu_n,\varphi\rangle-\langle\mu_{\Mand_d},\varphi\rangle|\leq C_1 \left(\frac{n}{d^n}\right)^{1/2}\|\varphi\|_{\mathcal{C}^1}~,\]
where $C_1$ depends only on $d$. Pick $w\in\D^*$ and $n\geq1$. We are left with estimating $|\langle\mu_n,\varphi\rangle-\langle\mu_{n,w},\varphi\rangle|$. For any $c_0\in\Per(n)$, let us denote by $\Omega_{c_0}$ the hyperbolic component containing $c_0$ and let $c_w^1,\ldots,c_w^{d-1}$ be the $d-1$ parameters $c_w^j\in\Omega_{c_0}\cap\Per(n,w)$. Let also $\Omega_{c_0}'$ be the open set 
\[\Omega_{c_0}':=\bigcap_{|t|<|w|}\Per(n,t)\cap \Omega_{c_0}\Subset\Omega_{c_0}~.\]
By Cauchy-Schwarz inequality,
\begin{eqnarray*}
|\langle\mu_n,\varphi\rangle-\langle\mu_{n,w},\varphi\rangle| & \leq & \frac{1}{(d-1)d^n}\sum_{c_0\in\Per(n)}\sum_{j=1}^{d-1}|\varphi(c_0)-\varphi(c_w^j)|\\
& \leq & \frac{1}{d^n}\sum_{c_0\in\Per(n)}\|\varphi\|_{\mathcal{C}^1}\textup{Diam}_\beta\left(\Omega_{c_0}'\right)\\
& \leq & \frac{1}{d^{n/2}}\|\varphi\|_{\mathcal{C}^1}\left(\sum_{c_0\in\Per(n)}\left(\textup{Diam}_\beta(\Omega_{c_0}')\right)^2\right)^{1/2}
\end{eqnarray*}
Recall that $\Per(n)=\bigcup_{k|n}\Per^*(k)$ where $\Per^*(k)$ denote the set of parameters admitting a super-attracting cycle of exact period $k$. Moreover, if $k|n$ and $c_0\in\Per^*(k)$,
\[\Omega_{c_0}'=\rho^{-1}\left(\D(0,|w|^{k/n})\right)~,\]
where $\rho:\Omega_{c_0}\longrightarrow\D$ is the map which, to $c$, associates the multiplier of the attracting cycle of $p_c$. Recall that $\rho$ is a $(d-1)$-branched cover ramifying exactly at $c_0$. In particular, 
\[\textup{mod}\left(\D\setminus\overline{\D(0,|w|^{k/n})}\right)=(d-1)\cdot\textup{mod}(\Omega_{c_0}\setminus\overline{\Omega_{c_0}'})~,\]
whence $\textup{mod}(\Omega_{c_0}\setminus\overline{\Omega_{c_0}'})=\frac{-k}{2\pi n(d-1)}\log|w|\geq \frac{-1}{2\pi n(d-1)}\log|w|$. By Lemma~\ref{lm:BriendDuval}, we deduce
\begin{eqnarray*}
|\langle\mu_n,\varphi\rangle-\langle\mu_{n,w},\varphi\rangle| & \leq & \frac{1}{d^{n/2}}\|\varphi\|_{\mathcal{C}^1}\left(\sum_{c_0\in\Per(n)}\frac{\tau}{\min(1,\frac{-1}{2\pi n(d-1)}\log|w|)}\textup{Area}_\beta(\Omega_{c_0})\right)^{1/2}.
\end{eqnarray*}
On the other hand, since $\Omega_{c_0}\cap\Omega_{c_0'}=\emptyset$ for $c_0\neq c_0'\in\Per(n)$ and $\Omega_{c_0}\subset\Mand_d\subset\overline{\D(0,2)}$, 
\[\sum_{\Per(n)}\textup{Area}_\beta(\Omega_{c_0})\leq \textup{Area}_\beta(\D(0,2))= 4\pi~.\]
Combined with the above, this gives
\begin{eqnarray*}
|\langle\mu_n,\varphi\rangle-\langle\mu_{n,w},\varphi\rangle| & \leq & \frac{2\sqrt{\pi\tau}}{d^{n/2}}\|\varphi\|_{\mathcal{C}^1}\max\left\{1,\frac{2\pi n(d-1)}{\log(|w|^{-1})}\right\}^{1/2},
\end{eqnarray*}
which, letting $C:=C_1+2\pi\sqrt{2(d-1)\tau}$, finally gives the wanted result.
\end{proof}

\section{Initial estimates in the moduli space of polynomials}\label{sec:outside}

We aim, here, at giving estimates in the spirit of the one provided by Lemma~\ref{lm:outside}. We begin the section with preliminaries on the Moebius function and dynatomic polynomials $P_{m,j}$. We then give estimates using Lemma~\ref{lm:Przytycki} for the potentials of the currents 
$[\Per^*_j(m)]=dd^c \log |P_{m,j}| $  (renormalized by their mass) at suitable parameters.

\subsection{Basics}\label{Sec:basics}

Let $d\geq3$. Recall that, for $0\leq j\leq d-2$ and $m\geq1$, we defined in \S~\ref{Sec:dynatomic} the polynomials $P_{m,j}$ of period $m$ for the critical point $c_j$ by
 $$P_{m,j}(c,a):=\prod_{k|m}\left(P_{c,a}^k(c_j)-c_j\right)^{\mu(m/k)}~,$$
where $\mu$ stands for the Moebius function, and that $P_{m,j}(c,a)=0$ if and only if $c_j$ is periodic under iteration of $P_{c,a}$ with exact period $m$. The degree of $P_{m,j}$ is equivalent to $d^{m}$. It does not depend on $j$ so we denote it by $d_m$ as we have the formula:
\begin{equation}\label{degree}
 d_m= \sum_{k|m} \mu(m/k) d^k=d^m + \sum_{k|m, \ k\neq m} \mu(m/k) d^k.
\end{equation}
Recall that the function $\sigma:\mathbb{N}^*\to \mathbb{N}^*$ is the \emph{sum of divisors}: 
We also let $\Per^*_j(m)$ be the algebraic variety
$$\Per^*_j(m):=\{(c,a)\in\C^{d-1} \, ; \ P_{m,j}(c,a)=0\}$$
defined by $(c,a)\in\Per^*_j(m)$ if and only if the critical point $c_j$ is periodic under iteration of $P_{c,a}$ with exact period $m$. We shall use in the sequel the following result (see \cite[Theorem 6.1]{favregauthier}).

\begin{theorem}[Favre-Gauthier]\label{tm:transverseFG}
Let $\underline{m}=(m_0,\ldots,m_{d-2})$ be a $(d-1)$-tuple of pairwise distinct positive integers such that $m_0\geq2$. If the hypersurfaces $\{\Per^*_j(m_j)\}_{0\leq j\leq d-2}$ intersect at $(c,a)\in\C^{d-1}$, their intersection at $(c,a)$ is smooth and transverse.
\end{theorem}
Let $\mathcal{H}_{\un}$ be the open set of hyperbolic parameters $(c,a)$ for which, for all $j$, the critical point $c_j$ is in the immediate basin of an attracting cycle of exact period $n_j$. For $(c_c,a_c)\in \cap_j \Per^*_j(n_j)$, we let $\Omega_{c_c,a_c}$ be the connected component of the interior of $\mathcal{C}_d$ that contains $(c_c,a_c)$ which is \emph{the center of the component $\Omega_{c_c,a_c}$}. In the case where all $n_j$ are distinct, it is known that the set $\Omega_{c_c,a_c}$ is simply connected and contains only one postcritically finite parameter, the parameter $(c_c,a_c)$. Notice that, according to Theorem \ref{tm:transverseFG} and to B\'ezout Theorem, the set $\mathcal{H}_{\un}$ has $\textup{Card}(\bigcap_j\Per^*_j(n_j))=\prod_jd_{n_j}$ distinct hyperbolic components.

\subsection{Przytycki's estimates in the space of polynomials}\label{sec:Przytycki}
We first give estimates for the functions $d^{-n}\log|P_{c,a}^{n}(c_j)-c_j|$ that will enable us to deal with the dynatomic polynomials. The two next lemmas are consequences of \cite[\S 6]{favredujardin} and of Lemma~\ref{lm:Przytycki}. These two lemmas intend to play the role, in the present setting, of Lemma~\ref{lm:outside}. Since the parameter space is now several dimensional, the proofs are more elaborate than in the unicritical family.

A few explanations are in order. In turn, lemma~\ref{magic_eq1} says that the growth of the function $d^{-n}\log|P_{c,a}^{n}(c_j)-c_j|$ is bounded above by the escape rate of the critical point $c_j$, up to exponential error term, at least when $c_j$ is the fastest escaping critical point. These estimates follow directly from classical estimates (see~\cite[S 6]{favredujardin}). To be more precise, we use the following two lemmas in the sequel (see \cite[Lemmas 6.4 $\&$ 6.5]{favredujardin}).

\begin{lemma}\label{lm:DF6.4}
There exists a constant $C>0$ depending only on $d$\footnote{observe that the statement of the lemma is incorrectly stated in \cite{favredujardin}, and the constant $C$ is actually \emph{independent} on $(c,a)$.} such that 
\[g_{c,a}(z)\leq \log^+\max\{|z|,|c|,|a|\}+C\]
for all $z\in\C$ and all $(c,a)\in\C^{d-1}$.
\end{lemma}

\begin{lemma}\label{lm:DF6.5}
For all $z\in\C$ and all $(c,a)\in\C^{d-1}$, we have
\[\max\{g_{c,a}(z),G(c,a)\}\geq \log\left|z-\delta\right|-\log4,\]
where $\delta=\sum_{j=1}^{d-2}c_j/(d-1)$.
\end{lemma}

The main difficulty is to get a \emph{uniform} constant in the error term.

On the other hand, lemma~\ref{magic_eq2} focuses on bounding from below the growth of $d^{-n}\log|P_{c,a}^{n}(c_j)-c_j|$ by the escape rate of the critical point $c_j$. Though the bound from below can not hold in full generality when $(c,a)\in\mathcal{C}_d$ (indeed, the critical point could be periodic for example), we manage to derive the wanted estimate when the critical point is active from Lemma~\ref{lm:Przytycki}.

\begin{lemma}\label{magic_eq1}
There exists a constant $C_1>0$ depending only on $d$ such that for any $(c,a)\in\C^{d-1}$ and any $j$ such that $g_{c,a}(c_j)=G(c,a)$, for any $n\geq1$, then
	\begin{equation*}
	\frac{1}{d^{n}}\log|P_{c,a}^{n}(c_j)-c_j|- g_{c,a}(c_j) \leq C_1 \frac{1}{d^{n}} ~.
	\end{equation*}
\end{lemma}

\begin{proof}
The proof breaks in two distinct parts: we first treat the case $G(c,a)=0$ and, in a second time, we focus on the case when $G(c,a)>0$ and $g_{c,a}(c_j)=G(c,a)$.

Assume first that $(c,a)\in \mathcal{C}_d$, i.e. $G(c,a)=0$. Then for any $j$, $g_{c,a}(c_j)=0=G(c,a)$.
Recall that $|P_{c,a}(z)|\geq \frac{1}{2} |z|^2$ as soon as $|z|\geq 2d^2\max \{|c|,|a|,1\}$ (see e.g. \cite{Ingram}). In particular, if $|P_{c,a}^n(c_j)|\geq 2d^2 \max \left\{|c|,|a|,2\right\}$ for some $n\geq1$ then $|P_{c,a}^{n+1}(c_j)| > 2|P_{c,a}^{n}(c_j)|$ and $|P_{c,a}^{n+1}(c_j)|\geq 2d^2 \max \left\{|c|,|a|,2\right\}$. Iterating the argument shows that the sequence $(P^m_{c,a}(c_j))_m$ diverges so $(c,a)\notin \mathcal{C}_d$. We deduce that for all $(c,a)\in \mathcal{C}_d$ and all $n\geq1$:
$$ \ |P^n_{c,a}(c_j)-c_j|\leq |P^n_{c,a}(c_j)|+|c_j|\leq 4d^2 \max_{(c,a)\in \partial\mathcal{C}_d}\left\{|c|,|a|,2\right\}. $$

In particular, there exists a constant $C$, independent of $n$ such that for all $(c,a)\in \mathcal{C}_d$ and all $n\geq1$:
\[\frac{1}{d^{n}} \log \left |P_{c,a}^{n}(c_j)-c_j\right| \leq \frac{C}{d^{n}}~. \]
Since $g_{c,a}(c_j)=0$ in $\mathcal{C}_d$, the lemma follows in the case where $(c,a)\in \mathcal{C}_d$. 

~

We now consider the case where $(c,a) \in \C^{d-1}\setminus \mathcal{C}_d$ and pick $R>0$ such that $(c,a)\in \{G=R\}$ and $g_{c,a}(c_j)=G(c,a)=R$. 

 According to Lemma~\ref{ineq:green}, since $G(c,a)=\log^+\max\left\{|c|,|a|\right\}+O(1)$, there exist $R_0>0$ such that for $R\geq R_0$, we can assume that for all $n\geq1$ : 
$$ |P^n_{c,a}(c_j)|\geq 2\max\{|c|,|a|\}\geq2|c_j|\geq 1.$$
Assume first that $R\geq R_0$. Hence:
\[\frac{1}{2} |P^{n}_{c,a}(c_j)| \leq |P^{n}_{c,a}(c_j) -c_j| \leq \frac{3}{2} |P^{n}_{c,a}(c_j)|\]
and if $\delta=(d-1)^{-1}\sum_lc_l$, then
\[\frac{1}{2} |P^{n}_{c,a}(c_j)| \leq |P^{n}_{c,a}(c_j) -\delta|~.\]
Taking the logarithm of the above inequality, and dividing it by $d^{n}$, and using $\log|P^{n}_{c,a}(c_j)|=\log^+|P^{n}_{c,a}(c_j)|$, we find
\[\frac{1}{d^n}\log^+|P^{n}_{c,a}(c_j)|  -\frac{\log{2}}{d^{n}} \leq \frac{1}{d^{n}}\log^+|P^{n}_{c,a}(c_j) -c_j| \leq \frac{1}{d^{n}}\log^+|P^{n}_{c,a}(c_j)| +\frac{\log{\frac{3}{2}}}{d^{n}}.\]
By invariance, $g_{c,a}(c_j)  =  d^{-n}g_{c,a}(P_{c,a}^{n}(c_j))=G(c,a)$ so that $\max \{ g_{c,a}(P_{c,a}^{n}(c_j)), G(c,a) \}=g_{c,a}(P_{c,a}^{n}(c_j))$. Hence, by \cite[Lemma 6.5]{favredujardin} (see Lemma~\ref{lm:DF6.5}), this gives a constant $C''>0$ depending only on $d$ and $R_0$ such that
\begin{eqnarray*}
g_{c,a}(c_j) & = & d^{-n}g_{c,a}(P_{c,a}^{n}(c_j))\geq d^{-n}\left(\log|P_{c,a}^{n}(c_j)-\delta |-\log4\right)\\
& \geq & d^{-n}\left(\log|P_{c,a}^{n}(c_j)-c_j|-C''\right).
\end{eqnarray*}
Assume now that $0<R< R_0$. We treat two cases separately. Assume first that  $ |P^n_{c,a}(c_j)|\leq 1/2|c_j|$ so that $   |P^{n}_{c,a}(c_j) -c_j| \leq 3/2 |c_j|$. 
In particular, taking the logarithm of the above inequality, dividing it by $d^{n}$ and using that $\log \leq \log^+$, this implies:
$$\frac{1}{d^{n}} \log | P^{n}_{c,a}(c_j) -c_j| \leq \frac{1}{d^{n}} \log | c_j| + \frac{\log{\frac{3}{2}}}{d^{n}} \leq \frac{1}{d^{n}} G(c,a) + \frac{C''}{d^{n}} $$
where $C''>0$ is a constant that does not depend on $(c,a)$ nor $n$ and where we used that $G(c,a)=\log^+\max\{|c|,|a|\}+O(1)$. In particular, since $G\geq0$ and $n\geq0$,
$$\frac{1}{d^{n}} \log | P^{n}_{c,a}(c_j) -c_j| \leq g_{c,a}(c_j) + \frac{C''}{d^{n}}.$$
Assume now that $ |P^n_{c,a}(c_j)|\geq 1/2|c_j|$ so that $|P^{n}_{c,a}(c_j) -c_j| \leq 3/2 |P^{n}_{c,a}(c_j)|$. Proceeding as above gives:
$$\frac{1}{d^{n}}\log |P^{n}_{c,a}(c_j) -c_j| \leq \frac{1}{d^{n}}\log^+|P^{n}_{c,a}(c_j)| +\frac{\log{\frac{3}{2}}}{d^{n}}.$$
Using Lemma~\ref{ineq:green} and $G(c,a)=R<R_0$ implies that there is a constant that depends only on $d$ and $R_0$, but neither on $(c,a)$ nor on $n$ that we still denote $C''$ such that:
$$\frac{1}{d^{n}}\log |P^{n}_{c,a}(c_j) -c_j| \leq g_{c,a}(c_j) +\frac{C''}{d^{n}}.$$
This concludes the proof in the case where $0<R< R_0$. 
\end{proof}

The second lemma we will need is the following.
\begin{lemma}\label{magic_eq2}
Let $d\geq3$. Then there exists a constant $C_2\geq C_1>0$ depending only on $d$ such that for any $(c,a)\in \C^{d-1}$ and any $0\leq j\leq d-2$ with either $g_{c,a}(c_j)=G(c,a)>0$ or $(c,a)\in \mathcal{C}_d\cap \mathrm{supp}(T_j)$ and any $n\geq1$, we have
	\begin{equation*}
	\frac{1}{d^{n}}\log|P_{c,a}^{n}(c_j)-c_j|- g_{c,a}(c_j) \geq - C_1 \frac{n}{d^{n}} ~.
	\end{equation*}
\end{lemma}

\begin{proof}
As above, we first treat the case $G(c,a)=0$ and then focus on the case $G(c,a)>0$. Consider first the case where $(c,a)\in \mathcal{C}_d\cap \mathrm{supp}(T_j)$. Hence  $c_j$ is active at $(c,a)$ and, in particular, $c_j$ does not lie in an attracting basin. By Lemma \ref{lm:Przytycki}, there exist a universal constant $0<\kappa<1$ and a continuous function $C(c,a)>1$ such that:
$$|P_{c,a}^{n}(c_j)-c_j|\geq\frac{|P_{c,a}^{n}(c_j)-c_j|}{\sqrt{1+|P_{c,a}^{n}(c_j)|^2}\cdot\sqrt{1+|c_j|^2}}=d_{\p^1}(P_{c,a}^{n_j}(c_j),c_j)\geq\frac{\kappa}{C(c,a)^{n}}~.$$
Hence, if $C:=\max_{(c,a)\in\partial\mathcal{C}_d}C(c,a)>1$, then
\begin{eqnarray*}
\frac{1}{d^{n}}\log|P_{c,a}^{n}(c_j)-c_j| \geq -\frac{n}{d^{n}} \log \left(\frac{C}{\kappa}\right) +g_{c,a}(c_j)
\end{eqnarray*}
for any $n\in \N^*$, since $g_{c,a}(c_j)=0$ on $\partial\mathcal{C}_d$. This is the expected result in the case where $(c,a)\in \mathcal{C}_d\cap \mathrm{supp}(T_j)$.

~

Assume now that  $(c,a)$ satisfies $G(c,a)= g_{c,a}(c_j)=R>0$. Consider the compact domain $\Omega_R:= \{G \leq R\}$.
As above, since $G(c,a)=\log^+\max\{|c|,|a|\}+O(1)$, according to Lemma~\ref{ineq:green}, taking $R_0$ large enough, we can assume that for all $n\geq1$ and all $R\geq R_0$: 
\[ |P^n_{c,a}(c_j)|\geq 2\max\{|c|,|a|\}\geq 2|c_j|\geq 1.\]
Assume first that $R\geq R_0$. Hence:
\[ |P^{n}_{c,a}(c_j) -c_j|\geq \frac{1}{2} |P^{n}_{c,a}(c_j)| .\]
Up to increasing $R_0$, combining $G(c,a)=\log^+\max\{|c|,|a|\}+O(1)$ and $g_{c,a}(c_j)=d^{-n}g_{c,a}(P^n_{c,a}(c_j))$ with \cite[Lemma 6.4]{favredujardin} (see Lemma~\ref{lm:DF6.4}) gives for $(c,a)$ and for any $n\geq 1$:
\begin{equation}\label{lemme1_eq3}
d^{-n}\log|P_{c,a}^{n}(c_j)|= d^{-n}\log^+|P_{c,a}^{n}(c_j)|\geq g_{c,a}(c_j)+O(d^{-n})~,
\end{equation}
where the $O$ is independent of $n$ and $(c,a)$. This gives a constant $C''$ (independent of $R\geq R_0$) such that for all $n\geq 1$:
$$g_{c,a}(c_j)-C''(d^{-n}) \leq  \frac{1}{d^{n}}\log |P^{n}_{c,a}(c_j) -c_j|.$$

Assume now  $R< R_0$ and consider $\Omega_{R_0}= \{G \leq R_0\}$. 

\begin{claim}
There exists a constant $C>1$ depending only on $d$ and $R_0$ such that, for all $(c,a)\in\Omega_{R_0}$, all $n\geq1$ and all $0\leq j\leq d-2$ with $G(c,a)=g_{c,a}(c_j)>0$, we have
\[d_{\p^1}(P_{c,a}^{n}(c_j),c_j)\geq\frac{\kappa}{C^{n}}~.\]
\end{claim}

Let us first finish the proof of Lemma~\ref{magic_eq2}. As a consequence,
$$\frac{|P_{c,a}^{n}(c_j)-c_j|}{\sqrt{1+|P_{c,a}^{n}(c_j)|^2}\cdot\sqrt{1+|c_j|^2}}=d_{\p^1}(P_{c,a}^{n_j}(c_j),c_j)\geq\frac{\kappa}{C^{n}}~.$$
Taking the logarithm in tha above inequality, dividing it by $d^n$ and using Lemma~\ref{ineq:green} uniformly on $\Omega_{R_0}$ imply:
\begin{align*}
 -\frac{n}{d^n}\log \frac{C}{\kappa} &\leq \frac{1}{d^n}|P_{c,a}^{n}(c_j)-c_j| - \frac{1}{2d^n}\log \sqrt{1+|P_{c,a}^{n}(c_j)|^2} -\frac{1}{2d^n}\log\left(1+|c_j|^2\right)\\
                                        &\leq  \frac{1}{d^n}|P_{c,a}^{n}(c_j)-c_j| - g_{c,a}(c_j)+ C'\frac{1}{2d^n},
\end{align*}
where $C'$ is yet another constant that depends on $d$ and $R_0$ but neither on $(c,a)$ nor on $n$. This concludes the proof. 
\end{proof}

\begin{proof}[Proof of the Claim]
Let us first pick
\[C>\sup_{(c,a)\in\Omega_{R_0}}\left(\sup_{z\in\p^1}P^\#_{c,a}(z)\right)^2>1~.\]
We proceed by contradiction, assuming that for some $n\geq2$, $d_{\p^1}(P_{c,a}^{n}(c_j),c_j)<\frac{\kappa}{C^{n}}$.

As seen in the proof of Lemma~\ref{lm:Przytycki}, letting $\varepsilon:=d_{\p^1}(P_{c,a}^{n}(c_j),c_j)$ and proceeding as in the proof of the Claim of Section~\ref{sec:localestimates}, we get
\[P_{c,a}^n\left(\B(c_j,2\varepsilon)\right) \subset \B(c_j,2\varepsilon)~.\]
 By Brouwer fixed point Theorem, $P^n_{c,a}$ has a fixed point in $\overline{\B}(c_j,2\varepsilon)$ and this ball is contained in the Fatou set of $P_{c,a}$. Hence $\infty\in\overline{\B}(c_j,2\varepsilon)$, since $c_j$ lies in the attracting basin of $\infty$ of $P_{c,a}$, i.e.
\[\frac{1}{\sqrt{1+|c_j|^2}}\leq d_{\p^1}(\infty,c_j)\leq 2 d_{\p^1}(c_j,P^n_{c,a}(c_j))<\frac{2\kappa}{C^{n}}~.\]
This may be rephrased as $C^2< C^{2n}< 4\kappa^2 \left(1+C(R_0)^2\right)$ where $C(R_0)=\max_{(c,a)\in \Omega_{R_0}} |c_j|$. Up to increasing $C$, we may assume $C^2\geq 2\kappa^2\left(1+C(R_0)^2\right)$, which is a contradiction.
\end{proof}

\subsection{Estimates for the dynatomic polynomials}

In the sequel, we shall use the notation
\[g_{n,j}(c,a):=\frac{1}{d_n}\log|P_{n,j}(c,a)|~, \ (c,a)\in\C^{d-1}, \ n\geq2~.\]
The following two propositions are keystones to the proof of Theorem~\ref{tm:centers}. These estimates are direct consequences of Lemma~\ref{magic_eq1}, Lemma~\ref{magic_eq2} and the maximum principle.

\begin{proposition}\label{prop:estimateabove}
There exists a constant $C\geq1$ depending only on $d$ such that for all $n\geq1$, $(c,a)\in \C^{d-1}$ and any  $0\leq j\leq d-2$ such that $g_{c,a}(c_j)=G(c,a)$,
\[g_{n,j}(c,a)-g_{c,a}(c_j)\leq C\frac{\sigma(n)}{d_{n}}~.\]
In particular, for any $0\leq j\leq d-2$ if $(c,a)\in\mathcal{C}_d$, we have:
\[g_{n,j}(c,a)\leq C\frac{\sigma(n)}{d_{n}}~.\]
\end{proposition}

\begin{proof}
By definition of $P_{n,j}$, we have:
 $$\log\left|P_{n,j}(c,a) \right|=\sum_{k|n}\mu(n/k) \log \left|P_{c,a}^k(c_j)-c_j\right|~,$$
Hence, dividing the above by $d_{n}$ and using \eqref{degree}, we deduce:
 $$\frac{1}{d_{n}}\log\left|P_{n,j}(c,a) \right|-g_{c,a}(c_j)=\sum_{k|n} \frac{d^k}{d_{n}} \mu(n/k) \left(\frac{1}{d^k}\log \left|P_{c,a}^k(c_j)-c_j\right|-g_{c,a}(c_j)\right)~.$$
Suppose that either $g_{c,a}(c_j)=G(c,a)>0$ or $(c,a)\in \mathcal{C}_d\cap \mathrm{supp}(T_j)$. Then Lemmas~\ref{magic_eq1} and~\ref{magic_eq2} imply that:
   $$ \left| \frac{1}{d^k}\log \left|P_{c,a}^k(c_j)-c_j\right|-g_{c,a}(c_j)\right|\leq C\frac{k}{d^k},$$
	for all $k$.
Since $\mu(n_j/k)=\pm 1$ and by the definition of the $\sigma$ function, we have:
\begin{eqnarray}~\label{prop9}
\left|g_{c,a}(c_j)- \frac{1}{d_{n}}\log\left|P_{n,j}(c,a) \right| \right|\leq  C\frac{\sigma(n)}{d_{n}} .
\end{eqnarray}
This ends the proof in the case where $g_{c,a}(c_j)=G(c,a)>0$ or $(c,a)\in \mathcal{C}_d\cap \mathrm{supp}(T_j)$. 

~

Finally, suppose that $(c,a)\in \partial_S\mathcal{C}_d$. As the Shilov boundary $\partial_S \mathcal{C}_d$ of $\mathcal{C}_d$ is contained in $\cap_j \supp{T_j}$, for all $n \geq2$ and all $0\leq j\leq d-2$, we have
\[\frac{1}{d_{n}}\log\left|P_{n,j}(c,a) \right| \leq  C\frac{\sigma(n)}{d_{n}}.\]
that is, for all $(c,a)\in \partial_S\mathcal{C}_d$, all $n \geq2$ and all $0\leq j\leq d-2$,
\[\left|P_{n,j}(c,a) \right| \leq  \exp\left(C\sigma(n)\right).\]
By the maximum principle, the bound extends to $\mathcal{C}_d$, which gives
\[\left|P_{n,j}(c,a) \right| \leq  \exp\left(C\sigma(n)\right),\]
for any $(c,a)\in \mathcal{C}_d$, $n\geq2$ and $0\leq j\leq d-2$.  Taking the logarithm in the above inequality and dividing it by $d_n$ finishes the proof.
\end{proof}

Observe that in the above proof, we proved the following crucial estimate (see \eqref{prop9}):
\begin{proposition}\label{prop:estimates}
Let $C\geq1$ be the constant given by Proposition~\ref{prop:estimateabove} and let $n\geq1$. Assume that either $g_{c,a}(c_j)=G(c,a)>0$, or $(c,a)\in\mathcal{C}_d\cap\textup{supp}(T_j)$. Then
\[g_{n,j}(c,a)-g_{c,a}(c_j)\geq -C\frac{\sigma(n)}{d_{n}}~.\]
\end{proposition}

\subsection{Locating ``bad'' parameters}

We shall now give two basic consequences of the above estimates which will be crucial to our aim. Namely, the following is a consequence of Propositions~\ref{prop:estimateabove} and~\ref{prop:estimates}. For any $(d-1)$-tuple $\un=(n_0,\ldots,n_{d-2})$ of positive integers, let
\[\mathcal{B}_{\un}:=\bigcap_{j=0}^{d-2}\left\{(c,a)\in\C^{d-1}\, ; \ |g_{n_j,j}(c,a)-g_{c,a}(c_j)|>C\frac{\sigma(n_j)}{d_{n_j}}\right\}~,\]
where $C\geq1$ is the constant given by Proposition~\ref{prop:estimateabove}.
\begin{corollary}\label{cor:dansCd}
For any $(d-1)$-tuple $\un=(n_0,\ldots,n_{d-2})$ of positive integers, 
\[\mathcal{B}_{\un}=\bigcap_{j=0}^{d-2}\left\{(c,a)\in\C^{d-1}\, ; \ |P_{n_j,j}|<e^{-C\sigma(n_j)}\right\}\subset \mathcal{H}_{\un}\subset\mathcal{C}_d~.\]
\end{corollary}

\begin{proof}
Let us first prove that the set 
\[\mathcal{B}_{\un}=\bigcap_{j=0}^{d-2}\left\{|g_{n_j,j}(c,a)-g_{c,a}(c_j)|>C\frac{\sigma(n_j)}{d_{n_j}}\right\}\]
is contained in the connectedness locus $\mathcal{C}_d$. Indeed, for every $(c,a)\in\C^{d-1}\setminus \mathcal{C}_d$, there exists $0\leq j\leq d-2$ such that $G(c,a)=g_{c,a}(c_j)>0$ . By Propositions~\ref{prop:estimateabove} and \ref{prop:estimates}, we have
\[-C\frac{\sigma(n_j)}{d_{n_j}}\leq g_{n_j,j}(c,a)-g_{c,a}(c_j)\leq C\frac{\sigma(n_j)}{d_{n_j}}~,\]
so that $(c,a)\in\C^{d-1}\setminus\mathcal{B}_{\un}$. The same argument implies in fact that $\mathcal{B}_{\un}$ is contained in the interior of $\mathcal{C}_d$, using the case  $(c,a)\in\mathcal{C}_d\cap\textup{supp}(T_j)$ in Proposition \ref{prop:estimates} and $\partial\mathcal{C}_d=\bigcup_j\mathcal{C}_d\cap\textup{supp}(T_j)$. Now from Proposition~\ref{prop:estimateabove}, we have that for all $(c,a)\in \mathcal{C}_d$ and all $0\leq j\leq d-2$:
 \[g_{n_j,j}(c,a)\leq C\frac{\sigma(n_j)}{d_{n_j}}~.\]
As $g_{c,a}(c_j)=0$ in $\mathcal{C}_d$, this implies:
\[\mathcal{B}_{\un}=\bigcap_{j=0}^{d-2}\left\{g_{n_j,j}<-C\frac{\sigma(n_j)}{d_{n_j}}\right\}\cap \mathcal{C}_d~.\]
This can be rewritten as:
\[\bigcap_{j=0}^{d-2}\left\{|P_{n_j,j}|<e^{-C\sigma(n_j)}\right\}\cap \mathcal{C}_d =\bigcap_{j=0}^{d-2}\left\{ |g_{n_j,j}(c,a)-g_{c,a}(c_j)|>C\frac{\sigma(n_j)}{d_{n_j}}\right\}~.\]
Arguing as above, we see that the set $\bigcap_{j=0}^{d-2}\left\{|P_{n_j,j}|<e^{-C\sigma(n_j)}\right\}$ is contained in $\mathcal{C}_d $: if not, for some $(c,a)$ in that set such that $g_{c,a}(c_j)=G(c,a)>0$, we have
\[g_{n_j,j}(c,a)-g_{c,a}(c_j)<g_{n_j,j}(c,a)<-C\frac{\sigma(n_j)}{d_{n_j}},\]
which contradicts Proposition \ref{prop:estimates}.

It remains to prove that $\mathcal{B}_{\un}\subset\mathcal{H}_{\un}$. Take a stable component $U\subset\mathcal{C}_d$ which is not contained in $\mathcal{H}_{\un}$. Consider the plurisubharmonic function:
$$ \phi_\un: = \max_j ( \log |P_{n_j,j}| + C \sigma(n_j) ).$$
Then its Monge-Amp\`ere $(dd^c \phi_\un)^{d-1}$ is $0$ on $U$ by hypothesis (see e.g. \cite{Demailly}). On the other hand, by Proposition~\ref{prop:estimates}, it is non negative in $\partial U$ ($\partial U\subset\partial \mathcal{C}_d\subset \cup_j \mathcal{C}_d\cap\textup{supp}(T_j)$). The comparison principle of Bedford and Taylor \cite{BedfordTaylor} implies that it is non negative on $U$. Hence $U\cap \mathcal{B}_{\un}=\varnothing$.
\end{proof}

Here is a consequence of Corollary~\ref{cor:dansCd}.

\begin{corollary}\label{cor:biholodansHn}
For any $(d-1)$-tuple $\un=(n_0,\ldots,n_{d-2})$ of pairwise distinct positive integers with $n_0\geq2$ and any connected component $U$ of $\mathcal{B}_{\un}$, the map
\[P_{\un}:=(P_{n_0,0},\ldots,P_{n_{d-2},d-2}):\C^{d-1}\longrightarrow\C^{d-1}\]
is a biholomorphism from $U$ to the polydisk $\D_{\un}:=\prod_j\D(0,e^{-C\sigma(n_j)})$.
\end{corollary}

\begin{proof}
By Corollary~\ref{cor:dansCd} and by the maximum principle, $P_{\un}$ maps $U$ to $\D_{\un}$ surjectively. Moreover, by definition of $U$, the map $P_{\un}$ is also proper on $U$, hence is a finite branched cover from $U$ to $\D_{\un}$. Let $(c_c,a_c)$ be the unique postcritically finite parameter of $U$. By Theorem~\ref{tm:transverseFG}, the map $P_{\un}$ is a local biholomorphism at $(c_c,a_c)$ and, since $P_{\un}^{-1}\{0\}=\{(c_c,a_c)\}$, the degree of $P_{\un}$ is $1$.
\end{proof}

\section{Distribution of postcritically finite hyperbolic polynomials}\label{sec:principale}

In the present section, we prove Theorem \ref{tm:centers}. We adapt the strategy we used in the unicritical family $(p_c)_{c\in\C}$ in Section~\ref{sec:unicritical} to the present situation. We use all along the present section the notations introduced in Sections~\ref{Sec:basics} $\&$~\ref{sec:Przytycki}. 
We assume from now on that $\un=(n_0,\ldots,n_{d-2})$ is a $(d-1)$-tuple of pairwise distinct integers with $n_0\geq2$. We also let $D_{\un}>0$ be the integer 
\[D_{\un}:=\prod_{j=0}^{d-2}d_{n_j}~.\]

\subsection{Estimates along specific curves}

Recall that we denoted by $\mathcal{H}_{\un}$ the union of hyperbolic components of $\mathcal{C}_d$ intersecting $\bigcap_j\Per^*_j(n_j)$. Let us first introduce some notations. For $\delta\in\C$, $0\leq j\leq d-2$, we let
\[g_{n_j,j}^\delta:=\frac{1}{d_{n_j}}\log|P_{n_j,j}-\delta|~, \text{ and }\ T_{n_j,j}^\delta:=dd^cg_{n,j}^\delta~.\]
For the rest of the section, we also shall write for $0\leq j\leq d-2$,
\[g_j(c,a):=g_{c,a}(c_j)~, \ \ T_j:=dd^cg_j~, \ \ g_{n_j,j}:=\frac{1}{d_{n_j}}\log|P_{n_j,j}| \ \textup{ and } \ T_{n_j,j}:=dd^cg_{n_j,j}~.\]
The key step of our proof can be summarized in the following proposition.

\begin{theorem}\label{tm:sauvelavie}
Let $C\geq1$ be the constant given by Proposition~\ref{prop:estimates}. Let $0\leq j\leq d-2$ and, for all $0\leq \ell\leq j-1$, pick $\delta_\ell\in\D\left(0,\exp(-C\sigma(n_\ell))\right)$. There exists a constant $C'\geq1$ which depends only on $d$ such that for any $\varphi\in\mathcal{C}^2_c(\C^{d-1})$, one has
\[\left|\left\langle \bigwedge_{\ell<j}T_{n_\ell,\ell}^{\delta_{\ell}}\wedge\left(T_{n_j,j}-T_j\right)\wedge\bigwedge_{k>j}T_{n_k,k}, \varphi\right\rangle\right|\leq C'\|\varphi\|_{\mathcal{C}^2}\frac{\sigma(n_j)}{d_{n_j}}~.\]
\end{theorem}

\begin{proof}
Fix $0\leq j\leq d-2$. First, the support of $\bigwedge_{\ell<j}T_{n_\ell,\ell}^{\delta_{\ell}}\wedge\bigwedge_{k>j}T_{n_k,k}$ is the algebraic curve
\[S_j:=\bigcap_{\ell<j}\{P_{n_\ell,\ell}=\delta_\ell\}\cap\bigcap_{k>j}\{P_{n_k,k}=0\}~.\]
Pick $\varphi\in\mathcal{C}^2_c(\C^{d-1})$. Then
\begin{eqnarray*}
I_j(\varphi)& := &\left\langle \bigwedge_{\ell<j}T_{n_\ell,\ell}^{\delta_{\ell}}\wedge\left(T_{n_j,j}-T_j\right)\wedge\bigwedge_{k>j}T_{n_k,k}, \varphi\right\rangle\\
& = & \int (g_{n_j,j}-g_j)dd^c\varphi\wedge \bigwedge_{\ell<j}T_{n_\ell,\ell}^{\delta_{\ell}}\wedge\bigwedge_{k>j}T_{n_k,k}\\
 & = & \frac{d_{n_j}}{D_{\un}}\int_{S_j}(g_{n_j,j}-g_j)\, dd^c\varphi~.
\end{eqnarray*}
As $\varphi$ is $\mathcal{C}^2$, we have (up to a constant that depends on the choice of the $\mathcal{C}^2$-norm) that $\|\varphi\|_{\mathcal{C}^2}\omega \pm dd^c \varphi \geq 0$, where $\omega$ is the Fubini-Study form on $\p^{d-1}$, hence we can write 
\[|I_j(\varphi)|\leq  \|\varphi\|_{\mathcal{C}^2}\frac{d_{n_j}}{D_{\un}}\int_{S_j}|g_{n_j,j}-g_j|\, \omega~.\]
According to Corollary~\ref{cor:dansCd}, the set $E_j:=S_j\cap\{ |g_{n_j,j}-g_j|>C\sigma(n_j)/d_{n_j}\}$ is contained in $S_j\cap\mathcal{H}_{\un}$ and coincides with $\{(c,a)\in S_j\, ; |P_{n_j,j}|<e^{-C\sigma(n_j)}\}$. Hence
\[\int_{S_j }|g_{n_j,j}-g_j|\, \omega\leq \int_{E_j}|g_{n_j,j}|\,\omega+ C\frac{\sigma(n_j)}{d_{n_j}}\int_{S_j}\omega~.\]

By Corollary~\ref{cor:biholodansHn}, the curve $S_j$ is smooth in the open set $\mathcal{B}_{\un}$. By the maximum principle, $E_j$ is a finite union of $D_{\un}$ topological disks (see \S~\ref{Sec:basics}), which are bounded since contained in $\mathcal{C}_d$. We now let $u_j:=d_{n_j}g_{n_j,j}=\log|P_{n_j,j}|$. Then $dd^c(u_j|_{S_j})=\sum \alpha_{z_0}\delta_{z_0}$, for some collection $\{\alpha_{z_0}\}$ of positive integers, where the sum ranges over $\Per^*_{j}(n_j)\cap S_j$. Moreover, by B\'ezout, we see that the finite measure $dd^cu_j\wedge S_j$ has mass $D_{\un}$. In particular, $dd^c(u_j|_{S_j})=\sum\delta_{z_0}$.

Let now $\Omega_j$ be a connected component of $E_j$. According to Theorem~\ref{tm:L1estimatedim1}, there exists $C_1>0$ universal such that
\[ \int_{\Omega_j}|u_j|\omega\leq \int_{\Omega_j}|u_j|\beta\leq C_1\max\{\|u_j\|_{L^\infty(\partial\Omega_j)},1\}\textup{Area}_\beta(\Omega_j)~.\]
Now, since $g_{n_j,j}=-C\frac{\sigma(n_j)}{d_{n_j}}$ on the boundary $\partial\Omega_j$ of $\Omega_j$ in $S_j$, we have $\|u_j\|_{L^\infty(\partial\Omega_j)}=C\sigma(n_j)\geq1$ and we find
\[ \int_{\Omega_j}|g_{n_j,j}|\omega \leq C_1C\frac{\sigma(n_j)}{d_{n_j}}\textup{Area}_\beta(\Omega_j)~.\]
Since $\mathcal{C}_d\Subset\B(0,16\sqrt{d-1})$, see e.g. \cite[Proof of Proposition 6.3]{favredujardin}, we can find $C_2>0$ depending only on $d$ such that  $\beta\leq C_2 \omega$ on $\mathcal{C}_d$. In particular, $\textup{Area}_\beta(\Omega_j)\leq C_2\textup{Area}_\omega(\Omega_j)$. Summing up on the components $\Omega_j$ of $E_j$ gives
\[\int_{E_j}|g_{n_j,j}|\omega \leq C_2C_1C\frac{\sigma(n_j)}{d_{n_j}}\int_{S_j}\omega~.\]
Finally, by B\'ezout, the integral $\int_{S_j}\omega$ is equal the degree of the curve, i.e. $D_{\un}/d_{n_j}$ (see e.g. \cite{Chirka}). Summarizing what we did to now, we find
\[|I_j(\varphi)|\leq \|\varphi\|_{\mathcal{C}^2}\left(1+C_2C_1\right)C\frac{\sigma(n_j)}{d_{n_j}}~,\]
which gives the wanted result since $C':=(1+C_2C_1)C$ depends only on $d$.
\end{proof}

\subsection{Intermediate estimates}

We introduce some notations. For all $0\leq j\leq d-2$, we let
\[\tilde g_{n_j,j}:=\max\left\{g_{n_j,j},-2C\frac{\sigma(n_j)}{d_{n_j}}\right\}~, \ \text{ and } \ \widetilde{T}_{n_j,j}:=dd^c\tilde{g}_{n_j,j}~.\]
Let us set $\delta_j(\theta):=e^{-2C\sigma(n_j)+i\theta}$, $\theta\in\R$. It is classical that 
\[\tilde g_{n_j,j}=\frac{1}{2\pi d_{n_j}}\int_0^{2\pi}\log\left|P_{n_j,j}-\delta_j(\theta)\right|\,d\theta~, \ \text{ and } \ \widetilde{T}_{n_j,j}=\frac{1}{2\pi}\int_0^{2\pi}T_{n_j,j}^{\delta_j(\theta)}\, d\theta~.\]

We now want to prove the following technical step of the proof of Theorem~\ref{tm:centers}.
\begin{lemma}\label{lm:facile}
Let $C$ be the constant given by Proposition~\ref{prop:estimates}. Then, for any $j\geq1$, any $0\leq r\leq j-1$ and any $\varphi\in\mathcal{C}^2_c(\C^{d-1})$, we have
\[\left|\left\langle\bigwedge_{\ell<r}T_\ell\wedge(T_r-\widetilde{T}_{n_r,r})\wedge\bigwedge_{r<q<j}\widetilde{T}_{n_q,q}\wedge(T_{n_j,j}-T_j)\wedge\bigwedge_{k>j}T_{n_k,k},\varphi\right\rangle\right|\leq 4C\frac{\sigma(n_{r})}{d_{n_{r}}}\|\varphi\|_{\mathcal{C}^2}~.\]\end{lemma}

\begin{proof}
Let $j \geq 1$. Pick $\varphi\in\mathcal{C}^2_c(\C^{d-1})$ and let
\begin{eqnarray*}
L_j(\varphi) & := & \left\langle\left(\bigwedge_{\ell<r}T_\ell\right)\wedge(T_r-\widetilde{T}_{n_r,r})\wedge\left(\bigwedge_{r<q<j}\widetilde{T}_{n_q,q}\right)\wedge(T_{n_j,j}-T_j)\wedge\left(\bigwedge_{k>j}T_{n_k,k}\right),\varphi\right\rangle\\
 & = & \int(g_{r}-\tilde{g}_{n_{r},r})dd^c\varphi\wedge\left(\bigwedge_{\ell<r}T_\ell\right)\wedge\left(\bigwedge_{r<q<j}\widetilde{T}_{n_q,q}\right)\wedge(T_{n_j,j}-T_j)\wedge\left(\bigwedge_{k>j}T_{n_k,k}\right).
\end{eqnarray*}
As $\varphi$ is $\mathcal{C}^2$, up to a constant that depends on the choice of the $\mathcal{C}^2$-norm, we have $\|\varphi\|_{\mathcal{C}^2}\omega \pm dd^c \varphi \geq 0$, where $\omega$ is the Fubini-Study form on $\p^{d-1}$, hence we can write
\begin{eqnarray*}
|L_j(\varphi)| & \leq & \|\varphi\|_{\mathcal{C}^2}\int\left|g_r-\tilde{g}_{n_r,r}\right| \omega\wedge\left(\bigwedge_{\ell<r}T_\ell\right)\wedge\left(\bigwedge_{r<q<j}\widetilde{T}_{n_q,q}\right)\wedge T_{n_j,j}\wedge\left(\bigwedge_{k>j}T_{n_k,k}\right)\\
& & +\|\varphi\|_{\mathcal{C}^2}\int\left|g_r-\tilde{g}_{n_r,r}\right| \omega\wedge\left(\bigwedge_{\ell<r}T_\ell\right)\wedge\left(\bigwedge_{r<q<j}\widetilde{T}_{n_q,q}\right)\wedge T_j\wedge\left(\bigwedge_{k>j}T_{n_k,k}\right).
\end{eqnarray*}
Each of the $(1,1)$-currents $T_\ell$, $T_{n_\ell,\ell}$ and $\widetilde{T}_{n_\ell,\ell}$ having projective mass $1$, by B\'ezout, we find
\begin{eqnarray}
|L_j(\varphi)| \leq 2 \|\varphi\|_{\mathcal{C}^2}\sup\left|g_r-\tilde{g}_{n_r,r}\right|~,\label{eq:inegaliteW}
\end{eqnarray}
where the supremum is taken over the support of the current
\[W:=\left(\bigwedge_{\ell<r}T_\ell\right)\wedge\left(\bigwedge_{r<q<j}\widetilde{T}_{n_q,q}\right)\wedge(T_{n_j,j}+T_j)\wedge\left(\bigwedge_{k>j}T_{n_k,k}\right)~.\]
Owing to~\eqref{eq:inegaliteW}, the next lemma ends the proof.
\end{proof}

\begin{lemma}
Let $C>0$ be as above. For any $(c,a)\in \supp(W)$, we have
\[\left|\tilde{g}_{n_r,r}(c,a)-g_r(c,a)\right|\leq 2C\frac{\sigma(n_r)}{d_{n_r}}~.\]
\end{lemma}

\begin{proof}
On $\textup{supp}(W)$, we have $g_\ell=0$ for all $\ell<r$ and all $\ell\geq j$. Indeed, for such $l$, by definition of $W$ we are in the set where the critical point $c_l$ has bounded orbit (observe that if $(c,a)\in \textup{supp}{T}_{n_l,l}$ then the critical point has periodic hence bounded orbit). Moreover, for any $r<\ell<j$, we have $|P_{n_\ell,\ell}|=e^{-2C\sigma(n_\ell)}$. Let $(c,a)\in \supp(W)$ such that $g_r(c,a)<G(c,a)$. This implies the existence of $r<\ell<j$ such that $G(c,a)=g_\ell(c,a)>g_r(c,a)\geq 0$. Then, by Proposition~\ref{prop:estimates}:
\[-2C\frac{\sigma(n_\ell)}{d_{n_\ell}}-g_\ell(c,a)=g_{n_\ell,\ell}(c,a)-g_\ell(c,a)\geq -C\frac{\sigma(n_\ell)}{d_{n_\ell}}~,\]
which implies $g_\ell<0$. This is a contradiction so we know that $g_r(c,a)=G(c,a)$ for all $(c,a)\in \supp(W)$. 

Assume now that $(c,a)\in\supp(W)$ satisfies $g_r(c,a)=0$. By the above, $(c,a)\in\mathcal{C}_d$, hence by Proposition~\ref{prop:estimateabove}:
\[g_{n_r,r}(c,a)\leq C\frac{\sigma(n_r)}{d_{n_r}}   \quad \mathrm{so} \quad \tilde{g}_{n_r,r}(c,a)\leq C\frac{\sigma(n_r)}{d_{n_r}}  ~.\]
Whence, by definition of $\tilde{g}_{n_r,r}(c,a)$:
\[-2 C\frac{\sigma(n_r)}{d_{n_r}}\leq \tilde{g}_{n_r,r}(c,a)= \tilde{g}_{n_r,r}(c,a)-g_r(c,a)\leq C\frac{\sigma(n_r)}{d_{n_r}}~.\]
Assume finally that $(c,a)\in \supp(W)$ satisfies $G(c,a)=g_r(c,a)>0$. By Propositions~\ref{prop:estimateabove} and~\ref{prop:estimates}, we find
\[\left|g_{n_r,r}(c,a)-g_r(c,a)\right|\leq C\frac{\sigma(n_r)}{d_{n_r}}~.\]
 In particular, $g_{n_r,r}(c,a)=\tilde{g}_{n_r,r}(c,a)$ and the estimate follows.
\end{proof}

\subsection{Speed of convergence: proof of Theorem~\ref{tm:centers}}

We are now in position to prove Theorem~\ref{tm:centers}. Pick $\varphi\in \mathcal{C}_c^2(\C^{d-1})$. We want to estimate
\[\Lambda_\un(\varphi):=\left\langle \mu_{\un}-\mu_\bif,\varphi\right\rangle=\left\langle \bigwedge_{j=0}^{d-2}T_{n_j,j}-\bigwedge_{j=0}^{d-2}T_j,\varphi\right\rangle~.\]
We shall decompose $\mu_{\un}-\mu_\bif$ into several pieces using a classical finite telescoping sum argument:
\begin{eqnarray}
\mu_{\un}-\mu_\bif & = & \sum_{j=0}^{d-2}\left(\bigwedge_{\ell<j}T_\ell\wedge(T_{n_j,j}-T_j)\wedge\bigwedge_{k>j}T_{n_k,k}\right)=\sum_{j=0}^{d-2}S_j~.\label{eq:decomposition}
\end{eqnarray}
For $j\geq1$, using $T_j=\widetilde{T}_{n_j,j}+(T_j-\widetilde{T}_{n_j,j})$, we rewrite the $j$-th term $S_j$ of \eqref{eq:decomposition} and find
\begin{eqnarray*}
S_j & = & \bigwedge_{\ell<j}\widetilde{T}_{n_\ell,\ell}\wedge(T_{n_j,j}-T_j)\wedge\bigwedge_{k>j}T_{n_k,k}\\
&  & +  \sum_{r<j-1}\left(\bigwedge_{\ell<r}T_\ell\wedge(T_r-\widetilde{T}_{n_r,r})\wedge \bigwedge_{r<q<j}\widetilde{T}_{n_q,q}\wedge(T_{n_j,j}-T_j)\wedge\bigwedge_{k>j}T_{n_k,k}\right)~.
\end{eqnarray*}
According to Lemma~\ref{lm:facile}, for $j\geq1$, we find
\[|\langle S_j,\varphi\rangle|\leq \left|\left\langle\bigwedge_{\ell<j}\widetilde{T}_{n_\ell,\ell}\wedge(T_{n_j,j}-T_j)\wedge\bigwedge_{k>j}T_{n_k,k},\varphi\right\rangle\right|+4C\sum_{r<j}\frac{\sigma(n_{r})}{d_{n_{r}}}\|\varphi\|_{\mathcal{C}^2}~.\]
We now want to give an estimate for
\[J_j:=\left\langle\bigwedge_{\ell<j}\widetilde{T}_{n_\ell,\ell}\wedge(T_{n_j,j}-T_j)\wedge\bigwedge_{k>j}T_{n_k,k},\varphi\right\rangle~.\]
To do so, we shall use the decomposition $\widetilde{T}_{n_\ell,\ell}=(2\pi)^{-1}\int_{[0,2\pi]}T_{n_\ell,\ell}^{\delta_\ell(\theta)}d\theta$ and Fubini:
\[J_j=\frac{1}{(2\pi)^{j-1}}\int_{[0,2\pi]^{j-1}}\left\langle\bigwedge_{\ell<j}T^{\delta_\ell(\theta_\ell)}_{n_\ell,\ell}\wedge(T_{n_j,j}-T_j)\wedge\bigwedge_{k>j}T_{n_k,k},\varphi\right\rangle d\theta_0\cdots d\theta_{j-1}~.\]
By Theorem~\ref{tm:sauvelavie}, there exists a constant $C'>0$ depending only on $d$ such that
\begin{eqnarray*}
|J_j| & \leq & \frac{1}{(2\pi)^{j-1}}\int_{[0,2\pi]^{j-1}}\left|\left\langle\bigwedge_{\ell<j}T^{\delta_\ell(\theta_\ell)}_{n_\ell,\ell}\wedge(T_{n_j,j}-T_j)\wedge\bigwedge_{k>j}T_{n_k,k},\varphi\right\rangle\right| d\theta_0\cdots d\theta_{j-1}\\
& \leq & \frac{1}{(2\pi)^{j-1}}\int_{[0,2\pi]^{j-1}}C' \frac{\sigma(n_j)}{d_{n_j}}\|\varphi\|_{\mathcal{C}^2} d\theta_0\cdots d\theta_{j-1}=C'\frac{\sigma(n_j)}{d_{n_j}}\|\varphi\|_{\mathcal{C}^2}~.
\end{eqnarray*}
All we have done so far summarizes as follows
\[\sum_{j=1}^{d-2}|\langle S_j,\varphi\rangle|\leq \sum_{j=1}^{d-2}\left(C'\frac{\sigma(n_j)}{d_{n_j}}+4C\sum_{r<j}\frac{\sigma(n_{r})}{d_{n_{r}}}\right)\|\varphi\|_{\mathcal{C}^2}~.\]
All which is left to do is to estimate $\langle S_0,\varphi\rangle$. By Theorem~\ref{tm:sauvelavie}, we have
\[\left|\langle S_0,\varphi\rangle\right|\leq C'\frac{\sigma(n_0)}{d_{n_0}}\|\varphi\|_{\mathcal{C}^2}~,\]
and we have finally proven
\[\left|\Lambda_\un(\varphi)\right|\leq \sum_{j=0}^{d-2}|\langle S_j,\varphi\rangle|\leq \left((d-1)C'+4C\sum_{j=1}^{d-2}(j-1)\right)\max_{0\leq j \leq d-2 }\left(\frac{\sigma(n_j)}{d_{n_j}}\right)\|\varphi\|_{\mathcal{C}^2}~,\]
which ends the proof, since $(d-1)C'+4C\sum_{j=1}^{d-2}(j-1)=(d-1)C'+2C(d-1)(d-2)$ depends only on $d$.

\section{Distribution of polynomials with prescribed multipliers}\label{sec:pluripolar}

The aim of this section is to derive Theorem~\ref{cor:centers} and Theorem~\ref{tm:principal} from Theorem~\ref{tm:centers}. We begin with the proof of Theorem~\ref{cor:centers} which is based on the same idea as that of \cite[Theorem 3]{favregauthier}.  For our purpose, we have to refine the techniques used in \cite{favregauthier}. We then prove Theorem~\ref{tm:principal} using DSH techniques (see e.g. \cite{ThelinVigny1}).

\subsection{Distribution of polynomials with $(d-1)$ attracting cycles}\label{sec:pernw}
Pick any $(d-1)$-tuple $\un=(n_0,\ldots,n_{d-2})$ of pairwise distinct positive integers with $n_0\geq2$. 

Let $\Omega$ be hyperbolic component of $\mathcal{C}_d$. Assume the for any $(c,a)\in\Omega$, the polynomial $P_{c,a}$ admits $(d-1)$ distinct attracting cycles $C_0,\ldots,C_{d-2}$ of  respective exact periods $n_0,\ldots,n_{d-2}$ and let $\rho_i(c,a)\in\D$ be the multiplier of the attracting cycle $C_i$. We call the map $\rho_\Omega:=(\rho_0,\ldots,\rho_{d-2}):\Omega\to\D^{d-1}$ the \emph{multipliers map} of the component $\Omega$ where $\D^{d-1}$ is the unit polydisk in $\C^{d-1}$. It is known that it is a biholomorphism (see \S~\ref{Sec:basics}).

~


Recall that we denoted $D_{\un}:=\prod_jd_{n_j}$ and that there exists polynomials 
\[p_n:\C^{d-1}\times\C\longrightarrow\C\]
detecting parameters having a cycle of exact period $n$ and given multiplier $w\in\C$ (see Lemma~\ref{lm:pernw} in \S~\ref{Sec:dynatomic}).
Moreover, we have
\begin{itemize}
\item $\deg_{c,a}p_n(\cdot,w)=(d-1)d_n$ for any $w\in\C$. Indeed, $\deg_{c,a}p_n(\cdot,w)$ is independent of $w$, we can write $[p_n(\cdot,0)=0]=\sum_j[\Per^*_j(n)]$ and $\Per^*_j(n)$ has degree $d_n$,
\item $\deg_{w}p_n(c,a,\cdot)=d_n/n$ for any $(c,a)\in\C^{d-1}$ (see \cite[\S 2.1]{BB2}).
\end{itemize}
For $w\in\C$ and $n\geq1$, let
\[\Per^*(n,w):=\{(c,a)\in\C^{d-1}\, ; \ p_n(c,a,w)=0\}.\]
Let $w_0,\ldots,w_{d-2}\in\D$ and let $n_1,\ldots,n_{d-2}$ be pairwise distinct positive integers with $n_0\geq2$. According to \cite[\S 6]{favregauthier}, for any $(c,a)\in\bigcap_j\Per^*(n_j,w_j)$, the intersection between the $\Per^*(n_j,w_j)$ is smooth and transverse at $(c,a)$ and $\textup{Card}(\bigcap_j\Per^*(n_j,w_j))$ is independent of $(w_0,\ldots,w_{d-2})$.
In particular, we have the following.

\begin{proposition}\label{prop:multip}
Let $w_0,\ldots,w_{d-2}\in\D$ and let $n_1,\ldots,n_{d-2}$ be pairwise distinct positive integers with $n_0\geq2$. Then $\textup{Card}(\bigcap_j\Per^*(n_j,w_j))=(d-1)!D_{\un}$~.
\end{proposition}
\begin{proof}
By the above, it is sufficient to estimate this cardinal for $w_0=\cdots=w_{d-2}=0$. For $n\geq1$, it is easy to see that $\Per^*(n,0)=\bigcup_{j=0}^{d-2}\Per^*_j(n)$. In particular,
\[\textup{Card}\left(\bigcap_{j=0}^{d-2}\Per^*(n_j,0)\right)=\sum_{s\in\mathfrak{S}_{d-1}}\textup{Card}\left(\bigcap_{j=0}^{d-2}\Per^*_{s(j)}(n_j)\right)~,\]
which ends the proof, since $\textup{Card}(\bigcap_{j=0}^{d-2}\Per^*_{s(j)}(n_j))=D_{\un}$ for any $s\in\mathfrak{S}_{d-1}$.
\end{proof}

 We are now in position to prove Theorem~\ref{cor:centers}. We deduce Theorem \ref{cor:centers} from Theorem~\ref{tm:centers}. Our strategy mixes arguments from the proof of Corollary~\ref{cor:Mandelbrot} and from the proof of \cite[Theorem 3]{favregauthier}.
 
\begin{proof}[Proof of Theorem~\ref{cor:centers}]
Let us first prove our result for $\uw=(0,\ldots,0)$. The set of parameters $(c,a)\in\C^{d-1}$ such that $P_{c,a}$ has $(d-1)$ distinct super-attracting cycles of respective exact periods $n_0,\ldots,n_{d-2}$ coincides with
\[\bigcup_{\sigma\in\mathfrak{S}_{d-1}}\bigcap_{j=0}^{d-2}\Per^*_{\sigma(j)}(n_j)~.\]
As the intersection $\bigcap_{j=0}^{d-2}\Per^*_{\sigma(j)}(n_j)$ is smooth and transverse for all $\sigma\in\mathfrak{S}_{d-1}$, the measure $\mu_{\un,\underline{0}}$ is the probability measure which is proportional to
\[\sum_{\sigma\in\mathfrak{S}_{d-1}}\frac{1}{D_{\un}}\bigwedge_{j=0}^{d-2}[\Per^*_{\sigma(j)}(n_j)]\]
which has mass $(d-1)!$. As a consequence, we have
\[\mu_{\un,\underline{0}}=\sum_{\sigma\in\mathfrak{S}_{d-1}}\frac{1}{(d-1)!D_{\un}}\bigwedge_{j=0}^{d-2}[\Per^*_{\sigma(j)}(n_j)]~.\]
Fix now $\varphi\in\mathcal{C}^1_c(\C^{d-1})$. By a classical interpolation argument, Theorem~\ref{tm:centers} gives a constant $C_1$ depending only on $d$, such that
\[|\langle\mu_{\un,\underline{0}},\varphi\rangle-\langle\mu_\bif,\varphi\rangle|\leq C_1 \left(\max_j\left(\frac{\sigma(n_j)}{d_{n_j}}\right)\right)^{1/2}\|\varphi\|_{\mathcal{C}^1}~.\]
This is the wanted result when $\uw=\underline{0}$.

~

Pick now $\uw\in\D^{d-1}\setminus\{\underline{0}\}$ and $\varphi\in\mathcal{C}^1_c(\C^{d-1})$. We write $\uw[0]\pe (0,\ldots,0)$ and for any $1\le j\le d-1$, we set $\uw[j]\pe (w_0,\ldots,w_{j-1},0,\ldots,0)$ and $\mu_j:=\mu_{\un,\uw[j]}$, so that $\uw[d-1]=\uw$, $\mu_0=\mu_{\un,\underline{0}}$ and $\mu_{d-1}=\mu_{\un,\uw}$. To conclude, it is sufficient to prove that
\[|\langle\mu_j,\varphi\rangle-\langle\mu_{j+1},\varphi\rangle|\leq C'\left(\max_{0\leq i\leq d-2}\left\{\frac{1}{d_{n_j}},\frac{-1}{d_{n_j}\log|w_i|}\right\}\right)^{1/2}\|\varphi\|_{\mathcal{C}^1},\]
for some constant $C'>0$ depending only on $d$. If $w_j=0$, there is nothing to prove, since in that case $\uw[j]=\uw[j+1]$. We thus may assume that $w_j\neq0$. we consider the algebraic subvariety 
$$C_j:=\bigcap_{h< j}\Per^*(n_h,w_h)\cap\bigcap_{l>j}\Per^*(n_l,0)~.$$
Observe that  $C_j \cap  \Per^*(n_j,0)=\bigcap_k\Per^*(n_k,\uw[j]_k)$ is finite, hence $C_j$ is an algebraic curve and that the intersections are smooth and transverse, by Proposition~\ref{prop:multip}.

Let $X_j:=\bigcap_k\Per^*(n_k,\uw[j]_k)$. Pick any point $(c,a) \in X_j$, and let $\Omega_{c,a}$ be the hyperbolic component containing $(c,a)$. Using \cite[Theorem 6.8]{favregauthier}, we define $\phi_{c,a,j} : \D(0, |w_j|^{-1/2}) \to \Omega_{c,a}$ by setting 
$$\phi_{c,a,j}(t) :=\rho_{c,a}^{-1} (w_0, \ldots , w_{j-1}, t w_j , 0, \ldots 0) ~.$$
By construction, the disks
\[\D_{c,a,j} \pe \phi_{c,a,j}( \D(0, |w_j|^{-1/2})) \ \text{and} \ \D_{c,a,j}' \pe \phi_{c,a,j}( \D(0, 1))\Subset\D_{c,a,j}\]
are included in $\Omega_{c,a} \cap C_j$, 
$\phi_{c,a,j}(0) = (c,a)$ and $\phi_{c,a,j} (1)$ belongs to $X_{j+1}$.
Any hyperbolic component contains a unique point in $X_j$, hence
the collection of disks $\D_{c,a,j}$ is disjoint.
Note also  that  any point in $X_{j+1}$ belongs to a hyperbolic component, and thus
is equal to  $\phi_{c,a,j} (1)$ for a unique $(c,a) \in X_j$.

By Cauchy-Schwarz inequality,
\begin{eqnarray*}
|\langle\mu_j,\varphi\rangle-\langle\mu_{j+1},\varphi\rangle| & \leq & \frac{1}{(d-1)!\cdot D_{\un}}\sum_{X_j}|\varphi(\phi_{c,a,j}(0))-\varphi(\phi_{c,a,j}(1))|\\
& \leq & \frac{1}{(d-1)!\cdot D_{\un}}\|\varphi\|_{\mathcal{C}^1}\sum_{X_j}\textup{Diam}_\beta\left(\D_{c,a,j}'\right)\\
& \leq & \frac{\|\varphi\|_{\mathcal{C}^1}}{\sqrt{(d-1)!\cdot D_{\un}}}\left(\sum_{X_j}\left(\textup{Diam}_\beta(\D_{c,a,j}')\right)^2\right)^{1/2}.
\end{eqnarray*}
 Recall that $\rho_{c,a}$ is a biholomorphism for any $(c,a)\in X_j$. In particular, 
\[\textup{mod}\left(\D_{c,a,j}\setminus\overline{\D_{c,a,j}'}\right)=\textup{mod}(\D(0,|w_j|^{-1/2})\setminus\overline{\D})=\frac{-1}{4\pi}\log|w_j|~.\]
 By Lemma~\ref{lm:BriendDuval}, we deduce
\begin{eqnarray*}
|\langle\mu_j,\varphi\rangle-\langle\mu_{j+1},\varphi\rangle| & \leq & \frac{\|\varphi\|_{\mathcal{C}^1}}{\sqrt{(d-1)!\cdot D_{\un}}}\left(\sum_{(c,a)\in X_j}\frac{\tau}{\min(1,\frac{-1}{4\pi}\log|w|)}\textup{Area}_\beta(\D_{c,a,j})\right)^{1/2}~.
\end{eqnarray*}
On the other hand, by B\'ezout, we have
\begin{eqnarray*}
\deg(C_j) & = & \sum_{i=0}^{d-2}\sum_{\underset{\sigma(j)=i}{\sigma\in\mathfrak{S}_{d-1},}}\deg\left(\bigcap_{k\neq j}\Per^*_{\sigma(k)}(n_k)\right)=\sum_{i=0}^{d-2}\sum_{\underset{\sigma(j)=i}{\sigma\in\mathfrak{S}_{d-1},}}\prod_{k\neq j}d_{n_k}\\
& = &\sum_{i=0}^{d-2}\sum_{\underset{\sigma(j)=i}{\sigma\in\mathfrak{S}_{d-1},}}\frac{D_{\un}}{d_{n_j}}=(d-1)!\frac{D_{\un}}{d_{n_j}}.
\end{eqnarray*}
From which, since $\Omega_{c,a}\cap\Omega_{c',a'}=\emptyset$ for $(c,a)\neq (c',a')\in X_j$ and $\D_{c,a,j}\subset \Omega_{c,a}\cap C_j$, we deduce
\[\sum_{(c,a)\in X_j}\textup{Area}_\omega(\D_{c,a,j})\leq \textup{Area}_\omega(C_j)\leq \deg(C_j)=(d-1)!\frac{D_{\un}}{d_{n_j}}.\]
Now, since $\Omega_{c,a}\subset\mathcal{C}_d\subset\B(0,16\sqrt{d-1})$ (see the proof of Theorem~\ref{tm:centers}), there exists a constant $C_2>0$ such that $\beta\leq C_2\omega$ on $\Omega_{c,a}$, for all $(c,a)\in X_j$. Hence $\textup{Area}_\beta(\D_{c,a,j})\leq C_2\textup{Area}_\omega(\D_{c,a,j})$. Combined with the above, this gives
\begin{eqnarray*}
|\langle\mu_j,\varphi\rangle-\langle\mu_{j+1},\varphi\rangle| & \leq & \frac{16\sqrt{\pi\tau(d-1)C_2}}{\sqrt{D_{\un}}}\|\varphi\|_{\mathcal{C}^1} \left(\max\left\{1,\sqrt{\frac{-4\pi}{\log|w_j|}}\right\}\right) \sqrt{\frac{D_{\un}}{d_{n_j}}}\\
& \leq & C_3\|\varphi\|_{\mathcal{C}^1}\left(\max_{0\leq \ell\leq d-2}\left\{\frac{1}{d_{n_\ell}},\frac{-1}{d_{n_\ell}\log|w_\ell|}\right\}\right)^{1/2},
\end{eqnarray*}
where $C_3=32\pi\sqrt{\tau(d-1)C_2}$.
\end{proof}

\subsection{Convergence for multipliers outside a pluripolar set: Theorem~\ref{tm:principal}}
For the content of the present, we are inspired by the techniques of \cite{ThelinVigny1} (see also \cite{dinhsibony2}). 
Let $\omega$ be the Fubini-Study form on $\p^1$ and let $\pi_j:(\p^1)^k\to \p^1$ be
the canonical projection to the $j$-th factor. We denote $\omega_j:= \pi_j^*(\omega)$. 
We define a smooth probability measure $\Omega$ on  $(\p^1)^k $ by $\Omega:= \omega_1\wedge \dots \wedge \omega_k$. 

 First, recall some facts on DSH functions. Let $k\geq1$. Recall that a probability measure
$\nu$ in $(\p^1)^k$ has \emph{bounded quasi-potentials} (or is PB) if $\nu$ admits a negative quasi-potential 
$U$ (that is a negative $(k-1,k-1)$ current $U$ such that $dd^c U+ \Omega= \nu$ in the sense of current) such that $|\langle U, S\rangle| \leq C$ for any positive smooth form $S$ 
of bidegree $(1,1)$ and mass $1$ (\cite{DinhSibonyregular}). Such nmapping $S \mapsto \langle U, S\rangle$ can be extended to any positive 
closed current of bidegree $(1,1)$ and mass $1$  with the same bound $|\langle U, S\rangle| \leq C$ (using structural varieties in the space of currents).
An interesting example of PB measure is the tensor product of measures in $\p^1$ with bounded quasi-potentials (in which case, having bounded quasi-potential is equivalent to the fact that the quasi-potential is bounded as a qpsh function).

We say that a function $\varphi$ on $(\p^1)^{d-1}$ is DSH if, outside a pluripolar set, it can be written as a difference of qpsh functions (for example, $\varphi \in \mathcal{C}^2$). Let $DSH\left(\left(\p^1\right)^{d-1}\right)$ be the space of such functions. We write $dd^c \varphi=T^+-T^-$ where $T^\pm$ are positive closed currents of bidegree $(1,1)$.
Let $\nu$ be a PB measure on $(\p^1)^{d-1}$. Then, the following defines a norm on the space $DSH\left(\left(\p^1\right)^{d-1}\right)$:
$$\|\varphi\|_{\nu}:=  \|\varphi\|_{L^1(\nu)} +\inf \|T^\pm\|$$
where the infimum is on all the decompositions $dd^c \varphi=T^+-T^-$. It turns out that taking another PB measure $\nu'$ gives an equivalent norm on $DSH\left(\left(\p^1\right)^{d-1}\right)$ (see e.g. \cite[p. 283]{dinhsibony2}).

Let $\varphi$ be a $\mathcal{C}^2$ function on $(\p^1)^{d-1}$, in particular $\varphi$ is DSH and let $\un=(n_0,\ldots,n_{d-2})$ be a $(d-1)$-uple of pairwise distinct positive integers with $n_0\geq2$. Consider the function:
$$ \Phi^\varphi_\un:\uw=(w_0,\ldots,w_{d-2}) \in \C^{d-1} \mapsto \langle \varphi , \mu_{\un, \uw} \rangle$$
where 
\[\mu_{\un,\uw}:=\frac{1}{(d-1)!\prod_j d_{n_{j}}}\bigwedge_{j=0}^{d-2}[\Per^*(n_{j},w_j)]~.\]
This intersection is well defined outside an analytic (hence pluripolar) set $\mathcal{E}_\un$. So, if we consider the pluripolar set $\mathcal{E}:=\cup_\un \mathcal{E}_\un$, the  map $\Phi^\varphi_\un$ is well defined outside $\mathcal{E}$ for all $\un$. Adding a pluripolar set if necessary, we can assume that this stands for $w\notin \mathcal{E}$ and all $j$ then $w_j\neq 1$. 

~

For the rest of the subsection, we let $\nu$ be a smooth probability measure with support in $\D(0,1/2)^{d-1}$. Such a measure indeed exists and is PB, furthermore, we will be able to apply uniform estimate for $\nu$ using Theorem~\ref{cor:centers}.

The main result we need here is the following.

 \begin{lemma}\label{lm:DSH}
There exists a constant $C_\varphi$ depending only on $\varphi$ and on $d$ such that for all $\un$, the function $\Phi_\un^\varphi$ is DSH with:
\[\|\Phi^\varphi_\un - \langle \mu_\bif, \varphi \rangle \|_{\nu} \leq C_\varphi \max_j\left( \frac{1}{n_j} \right). \]
\end{lemma}
\begin{proof}
Let $\Pi_1$ (resp. $\Pi_2$) denote the canonical projection from $\p^{d-1}\times \left(\p^1\right)^{d-1}$ to $\p^{d-1}$ (resp. $\left(\p^1\right)^{d-1}$). We let also $\tilde{\pi}_j : \left(\p^1\right)^{d-1}\to \p^1 $ be the canonical projection to the $j$-th factor. Let $\omega_1$ (resp. $\omega_{d-1}$) denote the Fubini-Study form on $\p^1$ (resp. $\p^{d-1}$). 
Consider in $\p^{d-1}\times \left(\p^1\right)^{d-1}$ the (trivial extension of the) analytic set:
\[  \mathcal{P}(n,j):= \left \{ (c,a,\uw), \  (c,a)\in \Per^*(n,w_{j-1}) \right\} . \]
Then  $(c,a,\uw)\in \mathcal{P}(n,j)$ if and only if $p_{n}(c,a,w_{j-1})=0$ where the polynomials  $p_{n}$ where defined in the previous section. Recall that $\deg_{w}p_n=d_n/n$ and $\deg_{(c,a)}p_n=(d-1)d_n$ (see Section~\ref{sec:pernw}). 

In particular, the current of integration $[\mathcal{P}(n,j) ]$ is cohomologous to 
\[(d-1)d_n \Pi_1^*(\omega_{d-1})+ \frac{d_n}{n} (\tilde{\pi}_{j}\circ\Pi_2)^*(\omega_1)~.\]
Observe that the function $\Phi_\un^\varphi$ can be defined in $\C^{d-1} \setminus \mathcal{E}$ as the following slice (we refer to \cite[p. 280]{dinhsibony2} for slicing on currents):
\[  \Phi^\varphi_{\un}(\uw)=\left\langle  \Pi_1^*(\varphi) \frac{1}{(d-1)!\prod_j d_{n_{j}}} \bigwedge_{j=0}^{d-2}[\mathcal{P}(n_j,j+1)] , \Pi_2 , \uw \right\rangle  .\] 
Since the slicing commutes with the operator $dd^c$, we have
\[ dd^c \Phi^\varphi_\un = \left\langle  \Pi_1^*(dd^c \varphi) \wedge \frac{1}{(d-1)!\prod_j d_{n_{j}}} \bigwedge_{j=0}^{d-2}[\mathcal{P}(n_j,j+1)] , \Pi_2 , \uw \right\rangle  .\] 
Write $dd^c \varphi=T^+-T^-$. As $\Pi_1^*(T^\pm)$ is cohomologous to $\|T^\pm\|\cdot\Pi_1^*(\omega_{d-1})$, the mass of
 \[ \mathcal{T}^\pm_{\un}:=\left\langle  \Pi_1^*(T^\pm) \wedge \frac{1}{(d-1)!\prod_j d_{n_{j}}} \bigwedge_{j=0}^{d-2}[\mathcal{P}(n_j,j+1)] , \Pi_2 , \uw \right\rangle \] 
can be computed in cohomology. In particular, it is bounded from above:
\[\|\mathcal{T}^\pm_{\un}\|\leq C_d\cdot \|T^+\| \max_{0\leq j\leq d-2}\left(\frac{1}{n_j}\right)~,\]
for some constant $C_d$ that depends only $d$ ($C_d=(d-1)^{d}/(d-1)!$ works).

On the other hand, according to Theorem \ref{cor:centers} above, for all $ \uw\in\textup{supp}(\nu)$, we have
\[|\Phi^\varphi_\un(\uw) - \langle \mu_\bif, \varphi \rangle| \leq  C\cdot \|\varphi\|_{\mathcal{C}^1}\cdot\max_{0\leq j\leq d-2}\left(\frac{\sigma(n_j)}{d^{n_j}}\right)\leq C_\varphi\max_{0\leq j\leq d-2}\left(\frac{1}{n_j}\right)\]
where $C>0$ depends only on $d$ and $C_\varphi>0$ depends only on $d$ and $\varphi$. The result follows since the function $\Phi^\varphi_\un$ is DSH and
$dd^c\Phi^\varphi_\un=\mathcal{T}_\un^+-\mathcal{T}_\un^-$.
\end{proof}

We finish the proof of Theorem  \ref{tm:principal}. 

\begin{proof}[Proof of Theorem~\ref{tm:principal}]
We assume that the series 
$$\sum_k \max_j\left( \frac{1}{n_{j,k}} \right)$$
converges. Pick $\varphi\in\mathcal{C}^2_c(\C^{d-1})$. By Lemma~\ref{lm:DSH}, $\Phi^\varphi_{\un_k}-\langle \mu_\bif, \varphi \rangle\in L^1(\nu)$ and
\[\|\Phi_{\un_k}^\varphi-\langle \mu_\bif, \varphi \rangle\|_{L^1(\nu)} \leq C_\varphi \max_j\left( \frac{1}{n_{j,k}} \right)~.\]
Hence, it converges $\nu$-almost everywhere to $0$. As all DSH-norms are equivalent, we have that for another $PB$ measure $\nu'$, there exists a constant $C'>0$ depending only on $k$ such that
\[\|\Phi^\varphi_{\un_k}-\langle \mu_\bif, \varphi \rangle\|_{L^1(\nu')} \leq C'\max_j\left( \frac{1}{n_{j,k}} \right)\]
and we can apply the same argument: it converges to $0$, $\nu'$-a.e. 

Finally, when a set $E$ is of $\nu'$-measure $0$ for all PB measure $\nu'$, it is pluripolar. Indeed, its logarithmic capacity has to vanish. We deduce, in particular, that the sequence $\Phi_{\un_k}^\varphi-\langle \mu_\bif, \varphi \rangle$ converges to $0$ outside a pluripolar set of $(\p^1)^{d-1}$.

 By separability, we apply the same argument to a dense countable family in $\mathcal{C}_c(\C^{d-1})$ of $\varphi \in \mathcal{C}^2\left((\p^1)^{d-1}\right)$ . As a countable union of pluripolar sets is again pluripolar, we deduce that outside a pluripolar set the measure $\mu_{\un_k,\uw}$ converges to $\mu_\bif$. This ends the proof of Theorem~\ref{tm:principal}.
\end{proof}

\begin{remark} \rm
\begin{enumerate}
\item In the moduli space of quadratic \emph{rational maps} (which is biholomorphic to $\C^2$), the curves $\Per^*(3,1)$ and $\Per^*(2,-3)$ have a common irreducible component (see \cite{Milnor3}). Such a behavior is expected to be impossible in the moduli space of degree $d$ \emph{polynomials} $\mathcal{P}_d$. We even expect the exceptional set $\mathcal{E}$ appearing in Theorem~\ref{tm:principal} to be empty.
\item We also expect the convergence of Theorem~\ref{tm:principal} to hold for \emph{any} sequence $\un_k$ of $(d-1)$-tuple, though we don't know how to prove it.
\end{enumerate}
\end{remark}

We now give a quick proof of Corollary~\ref{cor:neutres}.

\begin{proof}[Proof of Corollary~\ref{cor:neutres}]
The Shilov boundary of $\D^{d-1}$ is exactly $(\mathbb{S}^1)^{d-1}$ and the pluricomplex Green function of $\D^{d-1}$, which is given by $g=\log^+\max_i\{|z_i|\}$, is continuous and $\nu:=(dd^cg)^{d-1}$ coincides with the Lebesgue measure of $(\mathbb{S}^1)^{d-1}$, hence does not give mass to pluripolar sets. By Theorem~\ref{tm:principal}, for $\nu$-a.e. $\uw$, the sequence $(\mu_{\un_k,\uw})_k$ converges to $\mu_\bif$.
\end{proof}

\section{Distribution of parametric preimages}\label{sec:preimages}
Let $(f_\lambda)_{\lambda\in \Lambda}$ be a holomorphic
family of rational maps of degree $d\geq 2$ on $\p^1$ with $\mathrm{dim}(\Lambda)=m$.
Let $\omega_\Lambda$ be a K\"ahler form on $\Lambda$. Assume that  $c_1,\ldots,  c_k$  
are  marked critical points and  let $T_1,\ldots, T_k$ be their respective 
bifurcation currents (see Section~\ref{sec:defbif}). 

For $1\leq j \leq k$, we let $v^j_n :\Lambda \to \p^1$ be the map 
defined by $v^j_n(\lambda)= f^n_\lambda(c_j(\lambda))$ and set $v_n:=(v_n^1,\dots, v_n^k)$.
Let $\omega_j$ and $\Omega$ be defined as in Section~\ref{sec:pluripolar}. 
Then \cite[Theorem 0.3]{Dujardin2012} asserts that there exists a pluripolar set $\mathcal{E}\subset (\p^1)^k$
such that if $(z_1, \ldots, z_k)\in (\p^1)^k
\setminus \mathcal{E}$, the following equidistribution statement holds:  
$$\frac{1}{d^{nk}}\left[ \{f_\lambda^n(c_1(\lambda)) = z_1\}\cap \cdots \cap \{f_\lambda^n(c_k(\lambda)) = z_k\}\right]\to
T_1\wedge \cdots  \wedge T_k.$$
Observe that the equidistribution statement can be rewritten as:
$$\frac{1}{d^{nk}} (v_n^1)^*(\delta_{z_1})\wedge \cdots \wedge (v_n^k)^*( \delta_{z_k})=  \frac{1}{d^{nk}} (v_n)^*(\delta_{(z_1,\dots,z_k)}) \to
T_1\wedge \cdots  \wedge T_k.$$
We want to give here some improved versions of that theorem with speed estimates. 

The proof is quite elementary. Recall that PB measures were defined in the previous section.
 Our precise result may be formulated as follows.   
\begin{theorem}
Assume that the family $(f_\lambda)_{\lambda\in \Lambda}$ is algebraic, i.e. that $\Lambda$ is a quasi-projective variety and let $m:=\dim(\Lambda)$. Let $K$ be a compact subset $K\Subset \Lambda $, then with the above notations and hypotheses we have:
\begin{enumerate}
\item  Let $\nu$ be a probability measure on $(\p^1)^k$ with bounded potentials, then there exists $C(K)>0$ such that for
any $\mathcal{C}^2$-form $\varphi$ of bidegree $(m-k,m-k)$  and support in $K$ and any $n\geq1$:
$$\left|\langle \frac{1}{d^{nk}}(v_n)^*(\nu)- 
T_1\wedge \cdots  \wedge T_k , \varphi \rangle \right|\leq C(K) \frac{1}{d^n} \|\varphi\|_{\mathcal{C}^2(K)}~.$$ 
\item Let $(a_n)$ be a sequence of positive number such that $\sum a_n <\infty$, then there exists a pluripolar 
set $\mathcal{E}\subset (\p)^k$ such that if $(z_1, \ldots, z_k)\in (\p^1)^k
\setminus \mathcal{E}$, then 
there exists $C(K)>0$ such that for any $\mathcal{C}^2$-form $\varphi$ of bidegree $(m-k,m-k)$  and support in $K$ and any $n\geq1$:
$$\left|\langle\frac{1}{d^{nk}} (v_n)^*(\delta_{(z_1,\dots,z_k)}) - 
T_1\wedge \cdots  \wedge T_k , \varphi \rangle \right|\leq C(K) \frac{1}{a_n d^n} \|\varphi\|_{\mathcal{C}^2(K)}~.$$
\end{enumerate}
\end{theorem}
\begin{proof}
Let us first prove the first point.  Recall that if we let $T_n^j= d^{-n} (v_n^j)^*(\omega)$, then $T_n^j-T_j =dd^c V_n^j$ where $V_n^j = O(d^{-n})$ 
on $K$ (see e.g. \cite{Dujardin2012}) hence using the same arguments than above, it is enough to prove the bound:
 \begin{equation*}
|\langle \frac{1}{d^{nk}}(v_n)^*(\nu)- 
T_n^1\wedge \cdots  \wedge T_n^k , \varphi \rangle | \leq  C(K) \frac{1}{d^n} \|\varphi\|_{\mathcal{C}^2(K)}
\end{equation*}
for any $\mathcal{C}^2$-form $\varphi$ of bidegree $(m-k,m-k)$  and support in $K$. 

Write $T_n^1\wedge \cdots  \wedge T_n^k=d^{-kn}(v_n)^*(\Omega)$ and let $U$ be a quasi-potential of $\nu$ (i.e. $dd^cU+\Omega=\nu$). Applying Stokes formula (Lemma~\ref{lm:Stokes}), we find
 \begin{align*}
\langle \frac{1}{d^{nk}}(v_n)^*(\nu)- 
T_n^1\wedge \cdots  \wedge T_n^k , \varphi \rangle &= \langle \frac{1}{d^{nk}}(v_n)^*( U) , dd^c\varphi \rangle \\
                                                   &=\langle  U, \frac{1}{d^{nk}}(v_n)_*(dd^c\varphi) \rangle.
\end{align*}
Now, as $\varphi$ is $\mathcal{C}^2$, we have $\|\varphi\|_{\mathcal{C}^2} \omega_\Lambda^{m-k+1} \pm dd^c\varphi \geq 0$. 
As we assumed that $\Lambda$ is a quasi-projective variety, the mass of the current
$d^{-n(k-1)}(v_n)_*(\omega_\Lambda^{m-k+1} )$ is uniformly bounded. The fact that $\nu$ is PB gives the wanted result. 

We now prove the second point. Let $(a_n)$ be a sequence of positive numbers such that $\sum a_n <\infty$. Observe
that the positive closed current $S=\sum_n a_n d^{-n(k-1)}(v_n)_*(\omega_\Lambda^{m-k+1} )$ is well defined and has finite mass. 
For $z \in (\p^1)^k$, we let $U_z$ be the Green quasi-potential of $\delta_z $ (see \cite{DinhSibonysuper}).
The complementary of the set of points $z\in (\p^1)^k$ such that $\langle U_z, S\rangle >- \infty$ is  a pluripolar set  we shall denote by $\mathcal{E}$. For any $z \notin \mathcal{E}$, following the above proof, we find
\begin{align*}
0\geq  \langle  U_z, \frac{1}{d^{nk}}(v_n)_*(dd^c\varphi) \rangle  &\geq \langle  U_z, \frac{1}{d^{nk}}\|\varphi\|_{\mathcal{C}^2}(v_n)_*( \omega_\Lambda^{m-k+1}) \rangle  \geq  \frac{1}{a_n d^{n}}\|\varphi\|_{\mathcal{C}^2} \langle  U_z, S \rangle.     
\end{align*}
This ends the proof.
\end{proof}

Removing the assumption on the algebraicity of the family, one can using the same techniques, prove the following.
\begin{proposition}
 Let $K$ be a compact subset $K\Subset \Lambda $. For $j\leq k$ let $\nu_j$ be a measure in $\p^1$ with bounded 
 potentials and let $\nu :=\nu_1 \otimes\dots \otimes \nu_k$. Then there exists $C(K)>0$ such that for any $\mathcal{C}^2$-form 
 $\varphi$ of bidegree $(m-k,m-k)$  and support in $K$ and any $n\geq1$:
$$\left|\langle \frac{1}{d^{nk}}(v_n)^*(\nu)- 
T_1\wedge \cdots  \wedge T_k , \varphi \rangle \right|\leq C(K) \frac{1}{d^n} \|\varphi\|_{\mathcal{C}^2(K)}~.$$ 
\end{proposition}
The proof is similar to the above one although we do not need to push forward by $v_n$ here and use instead that 
the sequence $(d^{-n}(v^j_n)^*(\nu_j))$ is locally uniformly bounded in mass.

~

As observed in \cite{Dujardin2012}, the question is parallel to the equidistribution of preimages of points
under a rational map or more generally under a holomorphic endomorphism of $\p^k$ (see e.g. \cite{lyubich:equi,briendduval}). In the case of such a single map, it is known that the exceptional set, i.e. the set of points for which one does not have equidistribution, is 
algebraic.

In the particular case of the unicritical family, let $v_n(c):=p^n_c(0)$, $c\in\C$. Using the classical proof of the equidistribution of centers of hyperbolic components, we can prove that the exceptional set is reduced to the point $\infty$ which is exceptional for any $p_c$: 
\begin{theorem}
The sequence $\left(\frac{1}{d^{n-1}} (v_n)^*(\delta_z)\right)_{n\geq1}$ converges to $\mu_{\Mand_d}$ in the weak sense of currents if and only if $z\in \C$, i.e. the exceptional set is $\mathcal{E}=\{\infty\}$.
\end{theorem}
\begin{proof}
We follow closely Levin~\cite{Levin1}. Indeed, take  $u_n(c)= \frac{1}{d^{n-1}} \log |v_n(c)-z|$ so that $\Delta u_n = \frac{1}{d^{n-1}} (v_n)^*(\delta_z)$.
Then $u_n$ is a sequence of subharmonic functions that are bounded in $L^1_{loc}$. Extracting a converging subsequence,
we have that the limit $u$ is equal to $g_{\Mand_d}$ outside the Mandelbrot set since $|v_n(c)-z|\simeq |v_n(c)|$ outside the connectedness locus. Similarly,
$u\leq g_{\Mand_d}$ elsewhere. 
As $g_{\Mand_d}$ is continuous and $ \mu_{\Mand_d}$ gives no mass to the boundary of the component of the interior of $\Mand_d$, its implies
that $u=g_{\Mand_d}$ (see~\cite{multipliers}).

Finally, remark that the polynomial $v_n$ extends at infinity of the parameter space by setting $v_n(\infty)=\infty$. We get this way a polynomial map $v_n:\p^1\to\p^1$ of degree $d^{n-1}$. Hence $(v_n)^*(\delta_\infty)=d^{n-1}\delta_\infty$ for any $n\geq1$, which concludes the proof.
\end{proof}

\begin{question} Let $\mathcal{E}$ be the set of points $(z_1,\ldots,z_k)\in (\p^1)^k$ where $d^{-nk}(v_n)^*(\delta_{(z_1,\ldots,z_k)})$ does not converge to $T_1\wedge \cdots  \wedge T_k$.
\begin{enumerate}
 \item Is $\mathcal{E}$ empty for families of rational maps with generic empty exceptional set?
 \item Is it reduced to $\bigcup_{j=1}^k(\p^1)^{j-1}\times\{\infty\}\times(\p^1)^{k-j}$ in any family of polynomials?
 \end{enumerate} 
\end{question}

\bibliographystyle{short}
\bibliography{biblio}
\end{document}